\UseRawInputEncoding
\documentclass[12pt,reqno]{article}

\usepackage{amsfonts}
\usepackage{mathrsfs}

\usepackage{amsmath,amsthm,amsfonts,amssymb,latexsym,mathrsfs,color}

\usepackage[colorlinks=true,
linkcolor=webblue, filecolor=webbrown, citecolor=webred ]{hyperref}
\definecolor{webblue}{rgb}{0,.5,0}
\definecolor{webred}{rgb}{0,.5,0}
\definecolor{webbrown}{rgb}{.6,0,0}

\newtheorem{thm}{Theorem}[section]
\newtheorem{lem}[thm]{Lemma}

\newtheorem{prop}[thm]{Proposition}
\newtheorem{conj}[thm]{Conjecture}

\theoremstyle{definition}

\newtheorem{rem}[thm]{Remark}

\numberwithin{equation}{section}

\title{Total positivity from the exponential Riordan arrays
 \thanks{Supported partially by the National Natural
Science Foundation of China (Nos. 11971206, 12022105) and the
Natural Science Foundation for Distinguished Young Scholars of
Jiangsu Province (No. BK20200048).
\newline\hspace*{5mm}
   {\it Email address:} bxzhu@jsnu.edu.cn (B.-X. Zhu)}}
\author{Bao-Xuan Zhu}
\date{\footnotesize School of Mathematics and Statistics,
         Jiangsu Normal University,
         Xuzhou 221116, PR China}

\begin{document}

\maketitle

\begin{abstract}
Log-concavity and almost log-convexity of the cycle index
polynomials were proved by Bender and Canfield [J. Combin. Theory
Ser. A 74 (1996)]. Schirmacher [J. Combin. Theory Ser. A 85 (1999)]
extended them to $q$-log-concavity and almost $q$-log-convexity.
Motivated by these, we consider the stronger properties total
positivity from the Toeplitz matrix and Hankel matrix.

By using exponential Riordan array methods, we give some criteria
for total positivity of the triangular matrix of coefficients of the
generalized cycle index polynomials, the Toeplitz matrix and Hankel
matrix of the polynomial sequence in terms of the exponential
formula, the logarithmic formula and the fractional formula,
respectively.

Finally, we apply our criteria to some triangular arrays satisfying
some recurrence relations, including Bessel triangles of two kinds
and their generalizations, the Lah triangle and its generalization,
the idempotent triangle and some triangles related to binomial
coefficients, rook polynomials and Laguerre polynomials.  We not
only get total positivity of these lower-triangles, and
$q$-Stieltjes moment properties and $3$-$q$-log-convexity of their
row-generating functions, but also prove that their triangular
convolutions preserve Stieltjes moment property. In particular, we
solve a conjecture of Sokal on $q$-Stieltjes moment property of rook
polynomials.
\end{abstract}
\tableofcontents
\section{Introduction}
%%%%%%%%%%%%%%%%%%%%%%%%%%%%%%%%%%%%%%

\subsection{Cycle index polynomials and Sheffer polynomials}
Let $x_1, x_2,\ldots$ be a sequence of nonnegative real numbers and
define $A_n(x_1,x_2,\ldots,x_n)$ and $P_n(x_1,x_2,\ldots,x_n)$ by
\begin{equation}\label{CPI+function}
\sum_{n\geq0}A_n(x_1,x_2,\ldots,x_n)t^n=\sum_{n\geq0}P_n(x_1,x_2,\ldots,x_n)\frac{t^n}{n!}=\exp\left(\sum_{j\geq1}x_j\frac{t^j}{j}\right).
\end{equation}
Then each $A_n(x_1,x_2,\ldots,x_n)$ is a polynomial in the variables
$x_j$ with $1\leq j\leq n$. This has a well-known combinatorial
significance \cite{Com74}: Let $\Sigma_n$ denote the symmetric group
and let $N_j(\sigma)$ be the number of $j$-cycles in the permutation
$\sigma$. Then

\begin{equation}
A_n(x_1,x_2,\ldots,x_n)=\frac{P_n(x_1,x_2,\ldots,x_n)}{n!}=\frac{1}{n!}\sum_{\sigma
\in \Sigma_n} x_1^{N_1(\sigma)}x_2^{N_2(\sigma)}\cdots
x_n^{N_n(\sigma)}.
\end{equation}
This $A_n(x_1,x_2,\ldots,x_n)$ is called {\it the cycle index
polynomial} generally due to P\'{o}lya \cite{Pol37}, although in
fact appearing in earlier work of Redfield \cite{Red27}. Bender and
Canfield \cite{BC96} proved that if the positive sequence
$(x_n)_{n\geq0}$ with $x_0=1$ is log-concave then
$(A_n(x_1,x_2,\ldots,x_n))_{n\geq0}$ is log-concave and
$(P_n(x_1,x_2,\ldots,x_n))_{n\geq0}$ is log-convex. In addition, it
was extended to the much stronger properties $q$-log-concavity and
$q$-log-convexity, {\it i.e.,} for $n\geq m$,
\begin{eqnarray*}
A_n(x_1,\ldots,x_n)A_m(x_1,\ldots,x_m)-A_{n+1}(x_1,\ldots,x_{n+1})A_{m-1}(x_1,\ldots,x_{m-1})&\in& \mathbb{N}[\mathcal {X}],\\
P_{n+1}(x_1,\ldots,x_{n+1})P_{m-1}(x_1,\ldots,x_{m-1})-P_{n}(x_1,\ldots,x_n)P_{m}(x_1,\ldots,x_m)&\in&
\mathbb{N}[\mathcal {X}],
\end{eqnarray*}
where $\mathcal {X}=\{x_1,x_2,\ldots\} \bigcup
\{x_jx_k-x_{j-1}x_{k+1}:1\leq j\leq k\}$ and $\mathbb{N}[\mathcal
{X}]$ denotes the ring of polynomials in the indeterminates
$\mathcal {X}$ with nonnegative integer coefficients, see Bender and
Canfield \cite{BC96} and Schirmacher \cite{Sch99}.

The model in (\ref{CPI+function}) also has a probabilistic
interpretation from compound Poisson, see \cite{Mik09}. Its
log-concavity and log-convexity are very significant in probability
and statistics since they play a crucial role in the class of
infinitely divisible distributions, one of the most important
probability distributions in both theory and applications, see
\cite{Han88,HS88,Sat99}.

In addition, the sequences and functions in (\ref{CPI+function}) are
closely related to Sheffer polynomials \cite{All79,AV70,Shef39}. A
polynomial $S_n(q)=\sum_{k=0}^nS_{n,k}q^k$ is called {\it the
Sheffer polynomial} if
\begin{equation}\label{Shef+seq1}
\sum_{n\geq0}S_n(q)\frac{t^n}{n!}=g(t)\exp\left(f(t)q\right),
\end{equation}
where $g(0)f'(0)\neq0$ and $f(0)=0$ (cf. \cite{RR78,Rot75}). It was
proved that many polynomials including Bernoulli polynomials, Euler
polynomials, Hermite polynomials, Laguerre polynomials,
Poisson-Charlier polynomials and Meixner polynomials are all Sheffer
polynomials. Sheffer polynomials can be studied from umbral
calculus, which has been used in different areas of mathematics,
like approximation theory, analysis, combinatorics, statistics, and
so on. For more details on the theory of umbral calculus and Sheffer
polynomials, we refer the reader to \cite{RR78,RK73}.

Motivated by the $q$-log-concavity  of
$(A_{n}(x_1,x_2,\ldots,x_n))_{n\geq0}$ and $q$-log-convexity of
$(P_{n}(x_1,x_2,\ldots,x_n))_{n\geq0}$ in the first paragraph, the
aim of this paper is to extend them to total positivity of matrices
in terms of the Sheffer model in (\ref{Shef+seq1}). Let us recall
total positivity of matrices in the following.

\subsection{Total positivity}
Total positivity of matrices is an important and powerful concept
that arises often in various branches of mathematics, see the
monographs \cite{Kar68,Pin10} for more details.

Let $M=[m_{n,k}]_{n,k\ge 0}$ be a matrix of real numbers. It is {\it
totally positive} ({\it TP} for short) if all its minors are
nonnegative. It is {\it totally positive of order $r$} ({\it TP$_r$}
for short) if all minors of order $k\le r$ are nonnegative. For a
sequence $\alpha=(a_k)_{k\ge 0}$, denote its Toeplitz matrix by
$\Gamma(\alpha)=[a_{i-j}]_{i,j\ge 0}$ and  Hankel matrix by
$H(\alpha)=[a_{i+j}]_{i,j\ge 0},$ which play an important role in
different fields.

The sequence $\alpha$ is called a {\it P\'olya frequency} (PF, for
short) sequence if its infinite Toeplitz matrix $\Gamma(\alpha)$ is
TP. PF sequences have many nice properties (Karlin~\cite{Kar68}). In
classical analysis, PF sequences are closely related to real
rootedness of polynomials and entire functions. For example, the
fundamental representation theorem for PF sequences states that a
sequence $a_0=1,a_1,a_2,\ldots$ of real numbers is PF if and only if
its generating function has the form
$$\sum_{n\ge 0}a_nz^n
=\frac{\prod_{j\ge 1}(1+\alpha_jz)}{\prod_{j\ge
1}(1-\beta_jz)}e^{\gamma z}$$ in some open disk centered at the
origin, where $\alpha_j,\beta_j,\gamma\ge 0$ and $\sum_{j\ge
1}(\alpha_j+\beta_j)<+\infty$, see Karlin~\cite[p.~412]{Kar68} for
instance. In particular, a finite sequence of nonnegative numbers is
PF if and only if its generating function has only real zeros
(\cite[p.~399]{Kar68}). We also call the function $\sum_{n\ge
0}a_nz^n$ a PF (resp. PF$_r$) function if $\Gamma(\alpha)$ is TP
(resp. TP$_r$). The sequence $\alpha$ is called {\it log-concave} if
$a_{k-1}a_{k+1}\le a_k^2$ for all $k\ge 1$. Clearly, a sequence of
positive numbers is log-concave if and only if $\Gamma(\alpha)$ is
TP$_2$. Thus it is natural to consider a question whether results
for log-concavity can be extended to P\'olya frequency property. We
refer the reader to Br\"and\'en \cite{Bra06,Bra14},
Brenti~\cite{Bre89,Bre94,Bre95}, Wang-Yeh \cite{WYjcta05} and Zhu
\cite{Zhu14} on the PF property in combinatorics.

The sequence $\alpha$ is a {\it Stieltjes moment} ({\it SM} for
short) sequence if it has the form
\begin{equation}\label{i-e}
a_k=\int_0^{+\infty}x^kd\mu(x),
\end{equation}
where $\mu$ is a non-negative measure on $[0,+\infty)$ (see
\cite[Theorem 4.4]{Pin10} for instance). It is well known that
$\alpha$ is a Stieltjes moment sequence if and only if its Hankel
matrix $H(\alpha)$ is TP \cite{Pin10}. The Stieltjes moment problem
is one of classical moment problems and arises naturally in many
branches of mathematics \cite{ST43,Wid41}. The sequence $\alpha$ is
called {\it log-convex} if $a_{k-1}a_{k+1}\ge a_k^2$ for all $k\ge
1$. Clearly, a sequence of positive numbers is log-convex if and
only if $H(\alpha)$ is TP$_2$. In consequence, SM property is much
stronger than log-convexity. In addition, log-convexity of many
combinatorial sequences can be extended to SM property. We refer the
reader to Liu and Wang \cite{LW07} and Zhu \cite{Zhu13} for
log-convexity and Wang and Zhu \cite{WZ16} and
Zhu\cite{Zhu19,Zhu191} for SM property.

For brevity, we use the following notation
\begin{itemize}
  \item [\rm (i)] $\mathbb{R}$: the set of all real numbers; \quad $\mathbb{R}^{\geq0}:$ the set of all
nonnegative real numbers;
  \item [\rm (ii)] $\mathbb{N}^{+}:$ the set of all
positive integer numbers, and $\mathbb{N}=\mathbb{N}^{+}\cup\{0\}$.
\end{itemize}

If $\textbf{x}=(x_i)_{i\in{I}}$ is a set of indeterminates, we
denote by $\mathbb{R}[\textbf{x}]$ (resp.
$\mathbb{R}[[\textbf{x}]]$) the ring of polynomials (resp. formal
power series) in the indeterminates $\textbf{x}$ with coefficients
in $\mathbb{R}$. A matrix $M$ with entries in
$\mathbb{R}[\textbf{x}]$ is called {\it\textbf{x}-totally positive}
({\it \textbf{x}-TP} for short) if all its minors are polynomials
with nonnegative coefficients in the indeterminates $\textbf{x}$ and
is called {\it\textbf{x}-totally positive of order $r$} ({\it
\textbf{x}-TP$_r$} for short) if all its minors of order $k\le r$
are polynomials with nonnegative coefficients in the indeterminates
$\textbf{x}$. A polynomial sequence
$(\alpha_n(\textbf{x}))_{n\geq0}$ in $\mathbb{R}[\textbf{x}]$ is
called an {\it \textbf{x}-Stieltjes moment} ($\textbf{x}$-SM for
short) sequence (we also call it {\it coefficientwise Hankel-total
positive}, see \cite{PSZ18}) if its associated infinite Hankel
matrix is $\textbf{x}$-totally positive. It is called {\it
\textbf{x}-log-convex } ($\textbf{x}$-LCX for short) if
$$\alpha_{n+1}(\textbf{x})\alpha_{n-1}(\textbf{x})-\alpha_n(\textbf{x})^2\in\mathbb{R}^{\geq0}[\textbf{x}]$$
for all $n\in \mathbb{N}^{+}$ and is called {\it strongly
\textbf{x}-log-convex } ($\textbf{x}$-SLCX for short) if
$$\alpha_{n+1}(\textbf{x})\alpha_{m-1}(\textbf{x})-\alpha_n(\textbf{x})\alpha_m(\textbf{x})\in\mathbb{R}^{\geq0}[\textbf{x}]$$
for all $n\geq m\geq1$. Clearly, an $\textbf{x}$-SM sequence is both
$\textbf{x}$-SLCX and $\textbf{x}$-LCX. Define an operator $\mathcal
{L}$ by
$$\mathcal {L}[\alpha_i(\textbf{x})]:=\alpha_{i-1}(\textbf{x})\alpha_{i+1}(\textbf{x})-\alpha_i(\textbf{x})^2$$
for $i\in \mathbb{N}^{+}$. Then the $\textbf{x}$-log-convexity of
$(\alpha_i(\textbf{x}))_{i\geq 0}$ is equivalent to $\mathcal
{L}[\alpha_i(\textbf{x})]\in\mathbb{R}^{\geq0}[\textbf{x}]$ for all
$i\in \mathbb{N}^{+}$. In general, we say that
$(\alpha_i(\textbf{x}))_{i\geq 0}$ is {\it
$k$-\textbf{x}-log-convex} if $\mathcal
{L}^m[\alpha_i(\textbf{x})]\in\mathbb{R}^{\geq0}[\textbf{x}]$ for
all $m\leq k$, where $\mathcal {L}^m=\mathcal {L}(\mathcal
{L}^{m-1})$. The sequence $(\alpha_i(\textbf{x}))_{i\geq 0}$ is
called {\it infinitely \textbf{x}-log-convex} if it is
$k$-$\textbf{x}$-log-convex for every $k \in \mathbb{N}$. See
\cite{But90,Kra89,Le90} for $\textbf{x}$-log-concavity,
\cite{Bre95,Sag921} for $\textbf{x}$-PF,
\cite{CWY11,LW07,Zhu13,Zhu14,Zhu182} for $\textbf{x}$-LCX or
$\textbf{x}$-SLCX, and \cite{PS19,PSZ18,WZ16,Zhu19,Zhu20,Zhu202} for
$\textbf{x}$-SM.

In \cite{PSZ18,Zhu202}, the authors proved many results for
coefficientwise total positivity and coefficientwise Hankel-total
positivity. In a certain sense, this paper can be viewed as to keep
the study in \cite{PSZ18,Zhu202}. We will present our results in the
following.

\subsection{Extension of the cycle index polynomials}
As an extension of log-concavity of the cycle index polynomials, we
present a result for total positivity as follows.
\begin{thm}\label{thm+Toeplitz}
Let $g(t)\in \mathbb{R}[[t]]$ and
$\textbf{x}=(x_i)_{i\geq1}\subseteq\mathbb{R}$. Define a generalized
cycle index polynomial sequence $(\mathcal
{A}_n(\textbf{x}))_{n\geq0}$ by
\begin{equation}
\sum_{n\geq0}\mathcal
{A}_n(\textbf{x})t^n=g(t)\exp\left(\sum_{j\geq1}x_j\frac{t^j}{j}\right).
\end{equation}
Assume that $g(t)$ is a PF$_{r}$ function. If there exists a
nonnegative sequence $(\lambda_i)_{i\geq1}$ such that
$x_n=\sum_{i\geq1}\lambda_i^n$ for $n\geq1$, then the Toeplitz
matrix $[\mathcal {A}_{i-j}(\textbf{x})]_{i,j\geq0}$ is TP$_r$.
\end{thm}

Let functions $f(t)\in\mathbb{R}[[t]]$ and $g(t)\in\mathbb{R}[[t]]$,
where $f(0)=0$ and $f'(0)g(0)\neq0$. Define {\it a
cycle-index-triangle} $\left[A_{n,k}(f,g)\right]_{n,k\geq0}$ and
{\it a Sheffer-triangle} $\left[S_{n,k}(f,g)\right]_{n,k\geq0}$ by
\begin{eqnarray}\label{f-Sheffer+function}
g(t)\exp\left(qf(t)\right):=1+\sum_{n\geq1}\left(\sum_{k=1}^nA_{n,k}(f,g)q^k\right)t^n=1+\sum_{n\geq1}\left(\sum_{k=1}^nS_{n,k}(f,g)q^k\right)\frac{t^n}{n!}.
\end{eqnarray}

Similarly, define a logarithmic triangle
$\left[L_{n,k}(f)\right]_{n,k\geq0}$ and a fractional triangle
$[\widetilde{L}_{n,k}(f)]_{n,k\geq0}$ by
\begin{eqnarray}
\log\left(1-qf(t)\right):&=&\sum_{n\geq1}-\left(\sum_{k=1}^nL_{n,k}(f)q^k\right)\frac{t^n}{n!},\\
\frac{1}{1-qf(t)}:&=&1+\sum_{n\geq1}\left(\sum_{k=1}^n\widetilde{L}_{n,k}(f)q^{k}\right)\frac{t^n}{n!}.
\end{eqnarray}

More and more combinatorial triangles are proved to be totally
positive, for example, the triangle $[A_{n,k}]$ in \cite{Bre95},
which satisfies the recurrence
$$A_{n,k}=x_nA_{n-1,k}+y_nA_{n-1-t,k-1}+z_nA_{n-t,k-1},$$ the Pascal
triangle \cite[p.137]{Kar68}, recursive matrices \cite{CLW15},
Riordan arrays \cite{CLW152,CW19}, the Jacobi-Stirling triangle
\cite{Monge12}, Delannoy-like triangles \cite{MZ16},
Catalan-Stieltjes matrices \cite{PZ16}, Narayana triangles of types
$A$ and $B$ \cite{WY18}, and the generalized Jacobi-Stirling
triangle \cite{Zhu14}. We refer the reader to \cite{Bre89,GV85,S90}
for more total positivity results in combinatorics. Let $\bar{f}(t)$
denote the compositional inverse of the function $f(t)$ (that is
$f(\bar{f}(t))=\bar{f}(f(t))=t$) and let $\bar{f}'(t)$ denote the
derivative function of $\bar{f}(t)$. We get the following result for
total positivity of triangular matrices.
\begin{thm}\label{thm+Sheffer+triangle}
Let $\{r_1,r_2\}\subseteq \mathbb{N}^{+}$ and $g(t)$ be a PF$_{r_1}$
function. If one of $f(t)$ and $1/\bar{f}'(t)$ is a PF$_{r_2}$
function and let $r=\min\{r_1,r_2\}$, then
\begin{itemize}
\item [\rm (i)]both
$\left[S_{n,k}(f,g)\right]_{n,k\geq0}$ and
$\left[S_{n,k}(f,g(f))\right]_{n,k\geq0}$ are TP$_r$;
\item [\rm (ii)]both
$\left[A_{n,k}(f,g)\right]_{n,k\geq0}$ and
$\left[A_{n,k}(f,g(f))\right]_{n,k\geq0}$ are TP$_r$;
\item [\rm (iii)]
$\left[L_{n,k}(f)\right]_{n,k\geq0}$ is TP$_r$;
\item [\rm (iv)]
$\left[\widetilde{L}_{n,k}(f)\right]_{n,k\geq0}$ is TP$_r$.
\end{itemize}

\end{thm}

Further, let functions
$f(t)=\sum_{i\geq1}f_i\frac{t^i}{i!}\in\mathbb{R}[[t]]$ and
$g(t)=\sum_{i\geq0}g_i\frac{t^i}{i!}\in\mathbb{R}[[t]]$, where
$f_1\neq0$ and $g_0\neq0$. For $n\geq1$, denote the Sheffer
polynomial by
$$S_n(f,g;q)=\sum_{k=1}^nS_{n,k}(f,g)q^k$$ and
$S_0(f,g;q)=1$. Similarly, define {\it a logarithmic polynomial}
$L_n(f;q)$  and {\it a fractional polynomial} $\widetilde{L}_n(f;q)$
by
\begin{eqnarray}
L_{n}(f;q):=\sum_{k=1}^nL_{n,k}(f)q^k,\quad
\widetilde{L}_{n}(f;q):=\sum_{k=1}^n\widetilde{L}_{n,k}(f)q^k
\end{eqnarray}
for $n\geq1$, where $L_{0}(f;q)=1$ and $\widetilde{L}_{0}(f;q)=1$.
For a polynomial $P_n(q)$ of degree $n$, we denote by
$P^*_n(q)=q^nP_{n}(1/q)$ its reciprocal polynomial.

For Hankel-total-positivity, we get the result as follows.
\begin{thm}\label{thm+Sheffer+polynomial}
Let $g(t)=\exp(\lambda f(t))$. If $1/\bar{f}'(t)$ is a PF$_{r}$
function, then
\begin{itemize}
 \item [\rm (i)]
 the Hankel matrices $[S_{i+j}(f,g;q)]_{i,j\geq0}$ and $[S_{i+j}^*(f,g;q)]_{i,j\geq0}$ is
 $(\lambda,q)$-TP$_r$;
 \item [\rm (ii)]
the Hankel matrix $[f_{i+j+1}]_{i,j\geq0}$ is TP$_r$.
 \end{itemize}
\end{thm}

%\begin{thm}\label{thm+Sheffer+polynomial+2}
%%Let $g(t)=(1+f(t))^{\gamma}\exp(\lambda f(t))$ and
%$1/\bar{f}'(t)=(1+t)\phi(t)$. If $\phi(t)$ is a PF$_r$ function,
%then
%\begin{itemize}
% \item [\rm (i)]
 %the Hankel matrices $[S_{i+j}(f,g;q)]_{i,j\geq0}$ and $[S_{i+j}^*(f,g;q)]_{i,j\geq0}$ is
 %$(\gamma,\lambda,q)$-TP$_r$;
% \item [\rm (ii)]
%the Hankel matrix $[f_{i+j+1}]_{i,j\geq0}$ is TP$_r$.
% \end{itemize}
%\end{thm}
Let $A=[a_{n,k}]_{n,k\ge 0}$ be an infinite matrix.  For $n\in
\mathbb{N}$, define the $A$-convolution
\begin{eqnarray}\label{a-c}
z_n=\sum_{k=0}^{n}a_{nk}x_ky_{n-k}.
\end{eqnarray}
We say that \eqref{a-c} preserves the SM property: if both
$(x_n)_{n\ge 0}$ and $(y_n)_{n\ge 0}$ are Stieltjes moment
sequences, then so is $(z_n)_{n\ge 0}$.

Using positive definiteness of the quadratic form in linear algebra,
P\'olya and Szeg\"o~\cite[Part VII, Theorem 42]{PS64} proved for
$n\in \mathbb{N}$ that the binomial convolution
$$z_n=\sum_{k=0}^{n}\binom{n}{k}x_ky_{n-k}$$
preserves the SM property in $\mathbb{R}$. Recently, Wang and the
author \cite{WZ16} got more triangular convolutions including the
Stirling convolution of the second kind, the Eulerian convolution
and so on, preserving the SM property. In addition, the next
sufficient condition for the triangular convolution preserving the
SM property was provided in \cite{WZ16}.

\begin{lem}\label{lem+conv}\emph{\cite{WZ16}}
For $n\in \mathbb{N},$ let $A_n(q)=\sum_{k=0}^{n}a_{n,k}q^k$ be the
$n$th row generating function of the matrix $A=[a_{n,k}]_{n,k}$.
Assume that $(A_n(q))_{n\ge 0}$ is a Stieltjes moment sequence for
any fixed $q\ge 0$. Then the $A$-convolution \eqref{a-c} preserves
the SM property.
\end{lem}
Clearly, if the sequence $(A_n(q))_{n\geq0}$ is $q$-SM, then
$(A_n(q))_{n\ge 0}$ is a Stieltjes moment sequence for any fixed
$q\ge 0$. In addition, it also implies $3$-$q$-log-convexity of
$(A_n(q))_{n\geq0}$ in terms of the next result.

\begin{lem}\emph{\cite{Zhu18,Zhu202}}\label{lem+3-q-log-convex}
For a polynomial sequence $(A_n(\textbf{x}))_{n\geq0}$, if the
Hankel matrix $[A_{i+j}(\textbf{x})]_{i,j\geq0}$ is
$\textbf{x}$-TP$_4$, then $(A_n(\textbf{x}))_{n\geq0}$ is
$3$-$\textbf{x}$-log-convex.
\end{lem}

Then we have the following results.
\begin{thm}\label{thm+Sheffer+convolution}
Let $g(t)=\exp(\lambda f(t))$. If $1/\bar{f}'(t)$ is a PF function,
then we have
\begin{itemize}
 \item [\rm (i)]
 both $(S_{n}(f,g;q))_{n\geq0}$ and $(S^*_{n}(f,g;q))_{n\geq0}$ are $(\lambda,q)$-SM and
 $3$-$(\lambda,q)$-log-convex;
 \item [\rm (ii)]
 $(f_{n+1})_{n\geq0}$ is SM and $3$-log-convex;
 \item [\rm (iii)]
$[S_{n,k}(f,g)]_{n,k}$ is $\lambda$-TP;
\item [\rm (iv)] all triangular matrices $[L_{n,k}(f)]_{n,k}$,
$[\widetilde{L}_{n,k}(f)]_{n,k}$,
$[(-1)^{n-k}S_{n,k}(\bar{f},1)]_{n,k},[(-1)^{n-k}L_{n,k}(\bar{f})]_{n,k},$
and $[(-1)^{n-k}\widetilde{L}_{n,k}(\bar{f})]_{n,k}$ are TP;
\item [\rm (v)]
the $A$-convolution preserves the SM property if $A$ is any of
triangles $[S_{n,k}(f,g)]_{n,k}$ with $\lambda\geq0$,
$[S_{n+1,k+1}(f,1)]_{n,k}$ and $[\widetilde{L}_{n,k}(f)]_{n,k}$;
\item [\rm (vi)]
the $m$-branched Stieltjes-type continued fractions
\begin{eqnarray*}
   \sum_{n\geq0}S_{n}(f,g;q)t^n
   & = &
   \cfrac{1}
         {1 \,-\, \alpha_{m} t
            \prod\limits_{i_1=1}^{m}
                 \cfrac{1}
            {1 \,-\, \alpha_{m+i_1} t
               \prod\limits_{i_2=1}^{m}
               \cfrac{1}
            {1 \,-\, \alpha_{m+i_1+i_2} t
               \prod\limits_{i_3=1}^{m}
               \cfrac{1}{1 - \cdots}
            }
           }
         }
%
%   \cfrac{1}{1 - \cfrac{\alpha_{k+m} t}
%            {\prod\limits_{i_1=1}^{m}
%               \Biggl( \vspace*{-1cm}\displaystyle
%                      1 - \cfrac{{\alpha_{k+m+i_1} t}}
%                               {\displaystyle \prod\limits_{i_2=1}^{m}
%               \biggl( 1 - \frac{\alpha_{k+m+i_1+i_2} t}
%                               {1 - \cdots}
%               \biggr)
%                               }
%               \Biggr)
%            }}
\end{eqnarray*}
and \begin{eqnarray*}
   \sum_{n\geq0}S^*_{n}(f,g;q)t^n
   & = &
   \cfrac{1}
         {1 \,-\, \beta_{m} t
            \prod\limits_{i_1=1}^{m}
                 \cfrac{1}
            {1 \,-\, \beta_{m+i_1} t
               \prod\limits_{i_2=1}^{m}
               \cfrac{1}
            {1 \,-\, \beta_{m+i_1+i_2} t
               \prod\limits_{i_3=1}^{m}
               \cfrac{1}{1 - \cdots}
            }
           }
         }
%
%   \cfrac{1}{1 - \cfrac{\alpha_{k+m} t}
%            {\prod\limits_{i_1=1}^{m}
%               \Biggl( \vspace*{-1cm}\displaystyle
%                      1 - \cfrac{{\alpha_{k+m+i_1} t}}
%                               {\displaystyle \prod\limits_{i_2=1}^{m}
%               \biggl( 1 - \frac{\alpha_{k+m+i_1+i_2} t}
%                               {1 - \cdots}
%               \biggr)
%                               }
%               \Biggr)
%            }}
\end{eqnarray*} with coefficients
\begin{eqnarray*}
(\alpha_{i})_{i\geq
m}&=&(\lambda+q,\underbrace{x_1,\ldots,x_m}_{m},\lambda+q,\underbrace{2x_1,\ldots,2x_m}_{m},\lambda+q,\underbrace{3x_1,\ldots,3x_m}_{m},\ldots),\\
(\beta_{i})_{i\geq
m}&=&(q\lambda+1,\underbrace{qx_1,\ldots,qx_m}_{m},q\lambda+1,\underbrace{2qx_1,\ldots,2qx_m}_{m},q\lambda+1,\underbrace{3qx_1,\ldots,3qx_m}_{m},\ldots)
\end{eqnarray*}
if $1/\bar{f}'(t)=\prod_{i=1}^m(1+x_it)$.
 \end{itemize}
\end{thm}

\begin{thm}\label{thm+Sheffer+convolution+2}
If $1/[(1+t)\bar{f}'(t)]$ is a PF function and
$g(t)=(1+f(t))^{\gamma}\exp(\lambda f(t))$, then we have
\begin{itemize}
 \item [\rm (i)]
both $(S_{n}(f,g;q))_{n\geq0}$ and $(S^*_{n}(f,g;q))_{n\geq0}$ are
$(\gamma,\lambda,q)$-SM and $3$-$(\gamma,\lambda,q)$-log-convex;
\item [\rm (ii)]
$(S_{n}(f,(1+f)^{\gamma};q))_{n\geq0}$ is $(\gamma,q)$-SM and
$3$-$(\gamma,q)$-log-convex;
 \item [\rm (iii)]
$(f_{n+1})_{n\geq0}$ is SM and $3$-log-convex;
 \item [\rm (iv)]
$[S_{n,k}(f,g)]_{n,k}$ is $(\gamma,\lambda)$-TP;
\item [\rm (v)]
$[(-1)^{n-k}S_{n,k}(\bar{f},1/g(\bar{f}))]_{n,k}$ is TP for
$\gamma\geq0$ and $\lambda\geq0$;
\item [\rm (vi)]
the convolution $z_n=\sum_{k=0}^{n}S_{n,k}(f,g)x_ky_{n-k}$ preserves
the SM property for $\gamma\geq0$ and $\lambda\geq0$.
 \end{itemize}
\end{thm}

In the following, we will present the proof of Theorem
\ref{thm+Toeplitz} in Section $2$. In Section $3$, using the theory
of exponential Riordan arrays, we will present some generalized
results for total positivity. Then we use them to prove Theorems
\ref{thm+Sheffer+triangle}, \ref{thm+Sheffer+polynomial},
\ref{thm+Sheffer+convolution} and \ref{thm+Sheffer+convolution+2}.
In Section $4$, we apply our results to some combinatorial triangles
including Bessel triangles of two kinds and their generalizations,
the Lah triangle and its generalization, the idempotent triangle and
some triangles related to binomial coefficients, rook polynomials
and Laguerre polynomials. We not only get total positivity of these
lower-triangles, and $q$-Stieltjes moment properties and
$3$-$q$-log-convexity of their row-generating functions, but also
prove that their triangular convolutions preserve Stieltjes moment
property.

\section{Total positivity of the Toeplitz matrix}

The next result will be used repeatedly, which follows from the
classical Cauchy-Binet formula.
\begin{lem}\label{lem+prod+TP}
The product of two $\textbf{x}$-TP$_r$ ($\textbf{x}$-TP, resp.)
matrices is still $\textbf{x}$-TP$_r$ ($\textbf{x}$-TP, resp.).
\end{lem}

 \textbf{The Proof of Theorem
\ref{thm+Toeplitz}:} Note for the convolution
$c_n=\sum_{k\geq0}a_kb_{n-k}$ for $n\geq0$ that we have the
decomposition
$$[c_{i-j}]_{i,j\geq0}=[a_{i-j}]_{i,j\geq0}[b_{i-j}]_{i,j\geq0}.$$
So, by Lemma \ref{lem+prod+TP}, the product of two PF$_r$ functions
is a PF$_r$ function. Since $g(t)$ is still a PF$_r$, for the
P\'olya frequency of order $r$ of the product of $g(t)$ and
$\exp\left(\sum_{j\geq1}x_j\frac{t^j}{j}\right)$, it suffices to
prove a stronger result that the Toeplitz matrix $[\mathcal
{A}_{i-j}(\textbf{x})]_{i,j\geq0}$ is TP when $g(t)=1$.

 By taking
derivative in $t$ on both sides of the equality
\begin{equation*} \sum_{n\geq0}\mathcal
{A}_n(\textbf{x})t^n=\exp\left(\sum_{j\geq1}x_j\frac{t^j}{j}\right),
\end{equation*}
we have
\begin{equation*} \sum_{n\geq1}n\mathcal
{A}_n(\textbf{x})t^{n-1}=\sum_{j\geq1}x_jt^{j-1}\sum_{n\geq0}\mathcal
{A}_n(\textbf{x})t^n.
\end{equation*}
This implies for $n\geq1$ that $$n\mathcal
{A}_n(\textbf{x})=\sum_{j=1}^nx_j\mathcal {A}_{n-j}(\textbf{x}).$$
This recurrence relation is closely related to the complete
symmetric function $h_n$ and the power sum symmetric function $p_n$,
which satisfy $$nh_n=\sum_{j=1}^np_jh_{n-j}.$$ Moreover, it is also
known that the ordinary generating function of $h_n$ can be written
as
$$\sum_{n\geq0}h_nt^n=\prod_{i\geq1}\frac{1}{1-\beta_i t}.$$
So $(h_n)_{n\geq0}$ forms a PF sequence for $\beta_i\geq0$ for
$i\geq 1$. In consequence, if there exists a sequence
$(\lambda_i)_{i\geq1}$ of nonnegative real numbers such that
$x_n=\sum_{i\geq1}\lambda_i^n$, then $\mathcal {A}_n(\textbf{x})$ is
a complete symmetric function of $(\lambda_i)_{i\geq1}$. So
$(\mathcal {A}_n(\textbf{x}))_{n\geq0}$ is a PF sequence. That is to
say that the Toeplitz matrix $[\mathcal
{A}_{i-j}(\textbf{x})]_{i,j\geq0}$ is TP. The proof is complete.

\section{Exponential Riordan arrays and total positivity}
\subsection{Definition and properties of exponential Riordan arrays}
The set of Sheffer polynomials forms a group called the {\it Sheffer
group}, which is isomorphic to the exponential Riordan group. First
let us recall some properties of the exponential Riordan array. An
{\it exponential Riordan array}~\cite{Bar11-2,DFR,DS}, denoted by
$R=[R_{n,k}]_{n,k}=\left(g(t),f(t)\right)$, is an infinite lower
triangular matrix whose exponential generating function of the $k$th
column is
$$\frac{g(t)f^k(t)}{k!}$$ for $k=0,1,2,\ldots$, where $g(0)f'(0)\neq0$ and $f(0)=0$.
That is to say for $n,k\geq0$ that
$$R_{n,k}=\frac{n!}{k!}[t^n]g(t)f^k(t).$$
Let $R_n(q)=\sum_{k=0}^nR_{n,k}q^k$ be the row-generating function
of $R$. Then we have
\begin{eqnarray}\label{generating funtion+GR}
\sum_{n\geq0}R_n(q)\frac{t^n}{n!}=g(t)\exp\left(qf(t)\right).
\end{eqnarray}
 The group law is then given by
\begin{eqnarray}\label{Riordan+product}
(g,f)*(h,\ell)=(g\times h(f),\ell(f)).
\end{eqnarray}
The identity for this law is $I =(1,t)$ and the inverse of $(g,f)$
is $(g,f)^{-1}=(1/(g(\overline{f})),\overline{f})$, where
$\overline{f}$ is the compositional inverse of $f$.

 An exponential Riordan array can also be defined by the
 next recurrence relation.
\begin{prop}\emph{\cite{DFR}}\label{prop+TP+production}
Let $[R_{n,k}]_{n,k\geq0}=\left(g(t),f(t)\right)$ be an exponential
Riordan array. Then there exist two sequences $(z_n)_{n\geqslant0}$
and $(a_n)_{n\geqslant0}$ such that
$$R_{0,0}=1,\ \ R_{n,0}=\sum_{i\geqslant0}i!z_iR_{n-1,i},\ \
R_{n,k}=\frac{1}{k!}\sum_{i\geqslant
k-1}i!(z_{i-k}+ka_{i-k+1})R_{n-1,i}$$ for $n,k\geqslant1$. In
particular,
$$Z(t)=\frac{g'(\bar{f}(t))}{g(\bar{f}(t))},\quad
A(t)=f'(\bar{f}(t)),$$ where $Z(t)=\sum_{n\geqslant0} z_nt^n$ and
$A(t)=\sum_{n\geqslant0} a_nt^n$.
\end{prop}

The next result follows immediately from Proposition
\ref{prop+TP+production}.
\begin{prop}\label{prop+k+production}
Let $[R_{n,k}]_{n,k\geq0}=\left(g(t),f(t)\right)$ be an exponential
Riordan array. Assume $\mathfrak{R}=[\mathfrak{R}_{n,k}]_{n,k}$,
where $\mathfrak{R}_{n,k}=R_{n,k}k!$. Then there exist two sequences
$(z_n)_{n\geqslant0}$ and $(a_n)_{n\geqslant0}$ such that
$$\mathfrak{R}_{0,0}=1,\ \ \mathfrak{R}_{n,0}=\sum_{i\geqslant0}z_i\mathfrak{R}_{n-1,i},\ \
\mathfrak{R}_{n,k}=\sum_{i\geqslant
k-1}(z_{i-k}+ka_{i-k+1})\mathfrak{R}_{n-1,i}$$ for $n,k\geqslant1$,
where
$$Z(t)=\sum_{n\geqslant0} z_nt^n=\frac{g'(\bar{f}(t))}{g(\bar{f}(t))},\quad
A(t)=\sum_{n\geqslant0} a_nt^n=f'(\bar{f}(t)).$$
\end{prop}

Associated to each exponential Riordan array
$R=\left(g(t),f(t)\right)$, there is a matrix
$P=(p_{i,j})_{i,j\geqslant0}$, called the {\it production matrix},
such that
$$\overline{R}=RP,$$
where $\overline{R}$ is obtained from $R$ with the first row
removed. Assume that $z_{-1}=0.$ Deutsch {\it et al.}~\cite{DFR}
obtained the production matrix
\begin{eqnarray}\label{PP}
P=[p_{i,j}]_{i,j\geqslant0}=\left[
  \begin{array}{ccccccc}
    z_0&a_0& \\
   1!z_1& \frac{1!}{1!}(z_0+a_1)&a_0&&\\
   2!z_2&\frac{2!}{1!}(z_1+a_2) &\frac{2!}{2!}(z_0+2a_1)&a_0&\\
   3!z_3 &\frac{3!}{1!}(z_2+a_3)&\frac{3!}{2!}(z_1+2a_2)&\frac{3!}{3!}(z_0+3a_1)&\ddots\\
\vdots &\vdots&\vdots&\vdots&\ddots \\
  \end{array}
\right],
\end{eqnarray}
where the elements
$$p_{i,j}=\frac{i!}{j!}(z_{i-j}+ja_{i-j+1}),$$ for $i,j\geqslant 0.$

\begin{rem}
By Proposition (\ref{prop+k+production}), for the array
$\mathfrak{R}$ there is a matrix $\mathcal {P}=(\mathcal
{P}_{i,j})_{i,j\geqslant0}$ with
$$\mathcal {P}_{i,j}=z_{i-j}+ja_{i-j+1},$$ for $i,j\geqslant 0$ such
that $$\overline{\mathfrak{R}}=\mathfrak{R}\mathcal {P}.$$
\end{rem}

\subsection{Total positivity of exponential
Riordan arrays} The following result for total positivity is a
special case of \cite[Theorems 9.4, 9.7]{PSZ18}.

\begin{lem}\label{lem+TP+production}\cite{PSZ18}
Let $M=[M_{n,k}]_{n,k\geq0}$ be a lower triangular matrix with
$M_{0,0}=1$ and $\overline{M}=M\mathcal {P}$. If $\mathcal {P}$ is
$\textbf{x}$-TP$_r$, then both the matrix $[M_{n,k}]_{n,k\geq0}$ and
the Hankel matrix $[M_{i+j,0}]_{i,j\geq0}$ are $\textbf{x}$-TP$_r$.
\end{lem}

We will apply Lemma \ref{lem+TP+production} to the exponential
Riordan array $[R_{n,k}]_{n,k\geq0}$ and its associated array
$[\mathfrak{R}_{n,k}]_{n,k\geq0}$. The following is obvious and we
omit its proof for brevity.
\begin{lem}\label{lem+factor+TP}
Let $c_n$ and $d_n$ be positive real numbers for $n\geq0$. Then a
matrix $M=[M_{n,k}]_{n,k\geq0}$ is TP$_r$ if and only if the matrix
$M(\textbf{c,d})=[c_nd_kM_{n,k}]_{n,k\geq0}$ is TP$_r$.
\end{lem}

\begin{thm} \label{thm+ERA+TP}
Assume that $g(t)$ is a PF$_{r_1}$ function. If one of $f(t)$ and
$1/\bar{f}'(t)$ is a PF$_{r_2}$, then exponential Riordan arrays
$\left(g(t)e^{\lambda f(t)},f(t)\right)$ and
$\left(g(f(t))e^{\lambda f(t)},f(t)\right)$ are $\lambda$-TP$_r$,
where $r=\min\{r_1,r_2\}$.
\end{thm}

\begin{proof}
We first claim that the total positivity of order $r$ of
$(g(t),f(t))$ implies the $\lambda$-total positivity of order $r$ of
$(g(t)e^{\lambda f(t)},f(t))$. This can be proved as follows. Let
$R=[R_{n,k}]_{n,k}=(g(t),f(t))$ and
$\widehat{R}=[\widehat{R}_{n,k}]_{n,k}=(g(t)e^{\lambda f(t)},f(t))$.
Then \begin{eqnarray}
\sum_{n}\sum_{k}R_{n,k}q^k\frac{t^n}{n!}&=&g(t)e^{q f(t)},\\
\sum_{n}\sum_{k}\widehat{R}_{n,k}q^k\frac{t^n}{n!}&=&g(t)e^{(q+\lambda)
f(t)},
\end{eqnarray}
which implies \begin{eqnarray*} \widehat{R}_{n,k}&=&\sum_{i\geq
k}R_{n,i}\binom{i}{k}{\lambda}^{i-k}.
\end{eqnarray*}
Obviously,
$$\widehat{R}=RB,$$ where $B=[\binom{n}{k}{\lambda}^{n-k}]_{n,k}$.
From the total positivity of the Pascal triangle
$[\binom{n}{k}]_{n,k}$, we immediately get that $B$ is $\lambda$-TP.
Thus, by Lemma \ref{lem+prod+TP}, it follows from $\widehat{R}=RB$
that the total positivity of order $r$ of $R$ implies the
$\lambda$-total positivity of order $r$ of $\widehat{R}$. Similarly,
we can deduce that the total positivity of order $r$ of
$(g(f(t)),f(t))$ implies the $\lambda$-total positivity of order $r$
of $(g(f(t))e^{\lambda f(t)},f(t))$.

In what follows we only need to prove the total positivity of order
$r$ of $(g(t),f(t))$ and $(g(f(t)),f(t))$.

 In terms of the rule (\ref{Riordan+product}) about the product of Riordan arrays, we have decompositions
\begin{eqnarray}
(g(t),f(t))&=&(g(t),t)(1,f(t)),\\
(g(f(t)),f(t))&=&(1,f(t))(g(t),t).
\end{eqnarray}
Thus, in order to demonstrate total positivity of $(g(t),f(t))$ and
$(g(f(t)),f(t))$, in terms of Lemma \ref{lem+prod+TP}, it suffices
to prove that $(g(t),t)$ and $(1,f(t))$ are TP$_r$.

 Let $G=[G_{n,k}]_{n,k\geq0}$, where $G_{n,k}=[t^n]g(t)t^k$.
It follows from Lemma \ref{lem+factor+TP} that the total positivity
of $G$ implies that of $(g(t),t)$. Obviously, $G= \Gamma(g)$ is the
Toeplitz matrix of the sequence $(g_n)_{n\geq0}$, which is TP$_r$.

In the following, we will prove that $(1,f(t))$ is TP$_r$ if one of
$f(t)$ and $1/\bar{f}'(t)$ is a P\'olya frequency function of order
$r_2$.

(i) Assume that $f(t)$ is a P\'olya frequency function of order
$r_2$.  Let $F=[F_{n,k}]_{n,k\geq0}$ with $F_{n,k}=[t^n]f^k(t)$. We
give two different proofs as follows.

\textbf{The first proof:} It follows from Lemma \ref{lem+factor+TP}
that the total positivity of $F$ implies that of $(1,f(t))$. Denote
$F$ by $[1,f(x),f^2(x),f^3(x),\ldots]$ and its submatrix of the
first $n$ columns by
$$[1,f(x),f^2(x),f^3(x),f^{n-1}(x)]=\left[
  \begin{array}{ccccc}
    C_{n,n} \\
   D_{\infty,n} \\
  \end{array}
\right].$$ Let $h(t)=f(t)/t=\sum_{n\geq0}h_nt^n$ and $\Gamma(h)$
denote the Toeplitz matrix of $(h_n)_{n\geq0}$. Then we have
$$\left[
  \begin{array}{ccccc}
   C_{n,n}&0 \\
    D_{\infty,n} &\Gamma(h^{n-1})\\
  \end{array}
\right]\left[
  \begin{array}{ccccc}
    I_{n}&0 \\
   0 &\Gamma(h)\\
  \end{array}
\right]=\left[
  \begin{array}{ccccc}
   C_{n,n}&0 \\
    D_{\infty,n} &\Gamma(h^{n})\\
  \end{array}
\right].$$ In addition, it is obvious that
$$\left[
  \begin{array}{ccccc}
   C_{n,n}&0 \\
    D_{\infty,n} &\Gamma(h^{n})\\
  \end{array}
\right] =\left[
  \begin{array}{ccccc}
   C_{n+1,n+1}&0 \\
    D_{\infty,n+1} &\Gamma(h^{n})\\
  \end{array}
\right].$$ Thus, we get
\begin{eqnarray}\label{product+F}
\left[
  \begin{array}{ccccc}
   C_{n,n}&0 \\
    D_{\infty,n} &\Gamma(h^{n-1})\\
  \end{array}
\right]\left[
  \begin{array}{ccccc}
    I_{n}&0 \\
   0 &\Gamma(h)\\
  \end{array}
\right]=\left[
  \begin{array}{ccccc}
   C_{n+1,n+1}&0 \\
    D_{\infty,n+1} &\Gamma(h^{n})\\
  \end{array}
\right].
\end{eqnarray}

In order to prove total positivity of order $r_2$ of $F$, it
suffices to show that of $\left[
  \begin{array}{ccccc}
    C_{n,n} \\
   D_{\infty,n} \\
  \end{array}\right]$ for all $n$.  Note that $
\left[
  \begin{array}{ccccc}
    I_{n}&0 \\
   0 &\Gamma(h)\\
  \end{array}
\right] $ and the Toeplitz matrix $\Gamma(h^{n})$ are TP$_{r_2}$
because $f(t)$ is a P\'olya frequency function of order $r_2$.
Applying Lemma \ref{lem+prod+TP} to (\ref{product+F}), by induction
on $n$, we have $$\left[
  \begin{array}{ccccc}
   C_{n,n}&0 \\
    D_{\infty,n} &\Gamma(h^{n-1})\\
  \end{array}
\right]$$ is TP$_{r_2}$ for all $n$. In particular, the submatrix
$\left[
  \begin{array}{ccccc}
    C_{n,n} \\
   D_{\infty,n} \\
  \end{array}\right]$ is TP$_{r_2}$
for all $n$. Thus $F$ is TP$_{r_2}$.

\textbf{The second proof:\footnote{\quad This method is actually the
same as that of \cite{CW19} for total positivity, where it was given
as a constructed method.}} Note that $F_{n,k}=[t^n]f^k(t)=[t^n]f(t)
f^{k-1}(t)$ for $k\geq1$. Let $f(t)=\sum_{n\geq0}f_nt^n$. Then we
have
\begin{eqnarray}\label{F+f}
F_{n,k}=\sum_{j=k}^nf_{n-j+1}F_{j-1,k-1}
\end{eqnarray}
for $n,k\geq1$. Let $\overrightarrow{F_n}$ denote the matrix
consisting of columns from $1$ to $n$ and $F_n$ denote the matrix
consisting of columns from $0$ to $n-1$ of $F$. It follows from
(\ref{F+f}) that
$$\overrightarrow{F_n}=\Gamma(f(t)/t)F_n.$$
Thus in terms of Lemma \ref{lem+prod+TP} and total positivity of
$\Gamma(f)$, we obtain that $F_n$ is TP$_{r_2}$ by induction on $n$.
So the matrix $F$ is TP$_{r_2}$.

 (ii) Assume that $1/\bar{f}'(t)$ is a P\'olya frequency
function of order of $r_2$.

For the exponential Riordan array $(1,f(t))$, by Proposition
\ref{prop+TP+production}, there exist two functions $Z(t)$ and
$A(t)$ such that
$$Z(t)=\frac{g'(\bar{f}(t))}{g(\bar{f}(t))}=0,\quad
A(t)=f'(\bar{f}(t)).$$ On the other hand, it follows from
$f(\bar{f}(t))=t$ that
$$f'(\bar{f}(t))=\frac{1}{\bar{f}'(t)},$$
where $\bar{f}'(t)$ is the derivative function of $\bar{f}(t)$.
Thus, we have
\begin{eqnarray}\label{A+f}
A(t)=\frac{1}{\bar{f}'(t)}.
\end{eqnarray}  In consequence, $A(t)$ is a P\'olya frequency
of order $r_2$ in terms of the P\'olya frequency of order $r_2$ of
$1/\bar{f}'(t)$ and (\ref{A+f}). So the Toeplitz matrix
$\Gamma(A)=[a_{i-j}]_{i,j\geq0}$ is TP$_{r_2}$.

Let $\Theta=\left[
  \begin{array}{cccccccc}
    0&1&& \\
    &0&2&\\
   & &0&3&&\\
   & &&0&4&\\
 &&&&\ddots&\ddots \\
  \end{array}
\right],$ $\Lambda=\left[
  \begin{array}{cccccccc}
    0!&&& \\
    &1!&&\\
   & &2!&&&\\
   & &&3!&&\\
 &&&&\ddots \\
  \end{array}
\right]$. Then the production matrix $P=(p_{i,j})_{i,j\geqslant0}$
of $(1,f(t))$ satisfies
$$p_{i,j}=\frac{i!}{j!}ja_{i-j+1},$$ for $i,j\geqslant 0$, i.e.,
\begin{eqnarray*}
P&=&\Lambda \Gamma(A)\Theta\Lambda^{-1}.
\end{eqnarray*}
In consequence, by Lemma \ref{lem+factor+TP}, we immediately get
that the production matrix $P$ is TP$_{r_2}$. Then by Lemma
\ref{lem+TP+production}, we have the exponential Riordan array
$(1,f(t))$ is TP$_{r_2}$. This completes the proof.
\end{proof}

\begin{rem}\label{rem+TP}
For an exponential Riordan array $(g(t),f(t))$, it is often called a
proper exponential Riordan array for $g(0)f'(0)\neq0$ and $f(0)=0$.
If we drop the restricted condition, then it is called {\it a
general exponential Riordan array}. From the proof of Theorem
\ref{thm+ERA+TP}, it is obvious that the general exponential Riordan
arrays $(g(t)e^{\lambda f(t)},f(t))$ and $(g(f(t))e^{\lambda
f(t)},f(t))$ are also $\lambda$-TP$_r$ if both $g(t)$ and $f(t)$ are
PF$_r$ functions.
\end{rem}

\begin{rem}
We can repeatedly use Theorem \ref{thm+ERA+TP} to get total
positivity of more exponential Riordan arrays whose $f(t)$ and
$g(t)$ are not both PF.
\end{rem}

\textbf{Proof of Theorem \ref{thm+Sheffer+triangle}:}

(i) By (\ref{generating funtion+GR}), using the exponential Riordan
array, we have
$$\left[S_{n,k}(f,g)\right]_{n,k\geq0}=(g(t),f(t)),\quad
\left[S_{n,k}(f,g(f))\right]_{n,k\geq0}=(g(f(t)),f(t)).$$ Thus by
Theorem \ref{thm+ERA+TP}, we get that both
$\left[S_{n,k}(f,g)\right]_{n,k\geq0}$ and
$\left[S_{n,k}(f,g(f))\right]_{n,k\geq0}$ are TP$_{r}$, where
$r=\min\{r_1,r_2\}$.

(ii) By (\ref{f-Sheffer+function}), we have
$S_{n,k}(f,g)=n!A_{n,k}(f,g)$ and
$S_{n,k}(f,g(f))=n!A_{n,k}(f,g(f))$. Then (ii) immediately follows
from (i) and Lemma \ref{lem+factor+TP}.

 (iii) Using the property of the partial complete Bell polynomial
\cite[Theorem A, p.140]{Com74}, we get
$$L_n(f;q)=\sum_{k=1}^n(k-1)!S_{n,k}(f,1)q^k.$$
So we have $L_{n,k}(f)=(k-1)!S_{n,k}(f,1).$  Then by Lemma
\ref{lem+factor+TP} and (i), we immediately get that
$\left[L_{n,k}(f)\right]_{n,k\geq0}$ is TP$_{r}$.

(iv) In what follows we will show that
$\left[\widetilde{L}_{n,k}(f)\right]_{n,k\geq0}$ is TP$_{r}$.

By taking derivative in $q$ of
\begin{eqnarray}
\log\left(1-qf(t)\right)&=&\sum_{n\geq1}-L_n(f;q)\frac{t^n}{n!}=\sum_{n\geq1}-\left(\sum_{k=1}^nL_{n,k}(f)q^k\right)\frac{t^n}{n!},
\end{eqnarray}
we get
\begin{eqnarray}
\frac{-f(t)}{1-qf(t)}&=&\sum_{n\geq1}-L'_n(f;q)\frac{t^n}{n!}=\sum_{n\geq1}-\left(\sum_{k=1}^nkL_{n,k}(f)q^{k-1}\right)\frac{t^n}{n!}.
\end{eqnarray}
Then we have
\begin{eqnarray}
\frac{qf(t)}{1-qf(t)}&=&\sum_{n\geq1}qL'_n(f;q)\frac{t^n}{n!}=\sum_{n\geq1}\left(\sum_{k=1}^nkL_{n,k}(f)q^{k}\right)\frac{t^n}{n!}.
\end{eqnarray}
This implies
\begin{eqnarray}
\frac{1}{1-qf(t)}&=&1+\sum_{n\geq1}qL'_n(f;q)\frac{t^n}{n!}=1+\sum_{n\geq1}\left(\sum_{k=1}^nkL_{n,k}(f)q^{k}\right)\frac{t^n}{n!}.
\end{eqnarray}
Because we define
\begin{eqnarray}
\frac{1}{1-qf(t)}:&=&1+\sum_{n\geq1}\widetilde{L}_n(f;q)\frac{t^n}{n!}=1+\sum_{n\geq1}\left(\sum_{k=1}^n\widetilde{L}_{n,k}(f)q^{k}\right)\frac{t^n}{n!},
\end{eqnarray}
we have the relations
 \begin{eqnarray}\label{eq+L+L}
\widetilde{L}_n(f;q)=qL'_n(f;q),&&
\widetilde{L}_{n,k}(f)=kL_{n,k}(f)=k!S_{n,k}(f,1)
\end{eqnarray}
for $n\geq1$. In consequence, by Lemma \ref{lem+factor+TP} and (i),
we immediately get that
$\left[\widetilde{L}_{n,k}(f)\right]_{n,k\geq0}$ is TP$_{r}$. This
completes the proof of Theorem \ref{thm+Sheffer+triangle}. \qed

\subsection{Total positivity of the Hankel matrix of the $0th$ column}
In what follows we will consider total positivity of the Hankel
matrix of the $0th$ column from the exponential Riordan array.

\begin{thm}\label{thm+Hakel}
Let $\varphi(t)$ be a PF$_r$ function. For an exponential Riordan
array $[R_{n,k}]_{n,k}$, let $Z(t)=(\nu+\omega t)\varphi(t)$ and
$A(t)=(a+bt+ct^2)\varphi(t)$. Then we have the following results.
\begin{itemize}
 \item [\rm (i)]
If the tridiagonal matrix \begin{equation*}\label{J-eq}
J=\left[
\begin{array}{cccccc}
\nu & a &  &  &\\
\omega & \nu+b & 2a &\\
 & \omega+c & \nu+2b &3a &\\
  & &\omega+2c & \nu+3b &4a &\\
& && \ddots&\ddots & \ddots \\
\end{array}\right]
\end{equation*} is $(a,b,c,\nu,\omega)$-TP$_r$, then the lower-triangular matrix $[R_{n,k}]_{n,k}$ and the Hankel matrix $[R_{i+j,0}]_{i,j\geq0}$
are $(a,b,c,\nu,\omega)$-TP$_r$;
 \item [\rm (ii)]
If $\{\nu,\omega,a,b,c\}\subseteq\mathbb{R}^{\geq0}$, $\nu\geq w$
and $b\geq a+c$, then the lower-triangular matrix $[R_{n,k}]_{n,k}$
and the Hankel matrix $[R_{i+j,0}]_{i,j\geq0}$ are TP;
 \item [\rm (iii)]
If $\{\nu,\omega,a,b,c\}\subseteq\mathbb{R}^{\geq0}$, $\nu\geq a$
and $b\geq \max\{a+c,a+\omega\}$, then the lower-triangular matrix
$[R_{n,k}]_{n,k}$ and the Hankel matrix $[R_{i+j,0}]_{i,j\geq0}$ are
TP;
\item [\rm (iv)]
If $\omega=c=0$ and $r\rightarrow \infty$, then the lower-triangular
matrix $[R_{n,k}]_{n,k}$ and the Hankel matrix
$[R_{i+j,0}]_{i,j\geq0}$ are $(a,b,\nu)$-TP.
\item [\rm (v)]
Let $\varphi(t)=\prod_{i=1}^m(1+x_it)$. If $\omega=c=0$ and $a=1$,
then we have the $m$-branched Stieltjes-type continued fraction
\begin{eqnarray*}
   \sum_{n\geq0}R_{n,0}t^n
   & = &
   \cfrac{1}
         {1 \,-\, \alpha_{m} t
            \prod\limits_{i_1=1}^{m}
                 \cfrac{1}
            {1 \,-\, \alpha_{m+i_1} t
               \prod\limits_{i_2=1}^{m}
               \cfrac{1}
            {1 \,-\, \alpha_{m+i_1+i_2} t
               \prod\limits_{i_3=1}^{m}
               \cfrac{1}{1 - \cdots}
            }
           }
         }
%
%   \cfrac{1}{1 - \cfrac{\alpha_{k+m} t}
%            {\prod\limits_{i_1=1}^{m}
%               \Biggl( \vspace*{-1cm}\displaystyle
%                      1 - \cfrac{{\alpha_{k+m+i_1} t}}
%                               {\displaystyle \prod\limits_{i_2=1}^{m}
%               \biggl( 1 - \frac{\alpha_{k+m+i_1+i_2} t}
%                               {1 - \cdots}
%               \biggr)
%                               }
%               \Biggr)
%            }}
\end{eqnarray*}
with coefficients $$(\alpha_{i})_{i\geq
m}=(\nu,\underbrace{x_1,\ldots,x_m}_{m},\nu+b,\underbrace{2x_1,\ldots,2x_m}_{m},\nu+2b,\underbrace{3x_1,\ldots,3x_m}_{m},\ldots).$$
 \end{itemize}
\end{thm}

\begin{proof}
(i) In terms of Lemma \ref{lem+TP+production}, it suffices to prove
that the production matrix of $[R_{n,k}]_{n,k}$ is
$(a,b,c,\nu,\omega)$-TP$_r$. For the exponential Riordan array
$[R_{n,k}]_{n,k}$ with
$$Z(t)=(\nu+\omega t)\varphi(t),\quad A(t)=(a+bt+ct^2)\varphi(t),$$
it follows from (\ref{PP}) that its production matrix
\begin{eqnarray}\label{PPP}
P&=&\Lambda \Gamma(Z)\Lambda^{-1}+\Lambda \Gamma(A)\Theta\Lambda^{-1}\nonumber\\
&=&\Lambda \Gamma(\varphi)\left(\Gamma(\nu+\omega t)+
\Gamma(a+bt+ct^2)\left[
  \begin{array}{cccccccc}
    0&1&& \\
    &0&2&\\
   & &0&3&&\\
   & &&0&4&\\
 &&&&\ddots&\ddots \\
  \end{array}
\right]\right)\Lambda^{-1}\nonumber \\
 &=&\Lambda
\Gamma(\varphi)\left[
\begin{array}{cccccc}
\nu & a &  &  &\\
\omega & \nu+b & 2a &\\
 & \omega+c & \nu+2b &3a &\\
  & &\omega+2c & \nu+3b &4a &\\
& && \ddots&\ddots & \ddots \\
\end{array}\right]\Lambda^{-1}.
\end{eqnarray}
Because the infinite Toeplitz matrix $\Gamma(\varphi)$ and the
tridiagonal matrix $J$ are $(a,b,c,\nu,\omega)$-TP$_{r}$, applying
Lemma \ref{lem+prod+TP} to (\ref{PPP}), we immediately have $P$ is
$(a,b,c,\nu,\omega)$-TP$_{r}$. This shows that (i) holds.

(ii) and (iii) are immediate from (i) because the conditions in (ii)
and (iii) are sufficient to establish total positivity of $J$ (see
\cite{CLW15,Zhu202} for instance).

(iv) If $\omega=c=0$, then the tridiagonal matrix $J$ reduces to an
upper-bidiagonal matrix$$\left[
\begin{array}{cccccc}
\nu & a &  &  &\\
 & \nu+b & 2a &\\
 &  & \nu+2b &3a &\\
  & & & \nu+3b &4a &\\
& && \ddots&\ddots & \ddots \\
\end{array}\right],$$ which is $(a,b,\nu)$-TP. By (i), we immediately get that both the
lower-triangular matrix $[R_{n,k}]_{n,k}$ and the Hankel matrix
$[R_{i+j,0}]_{i,j\geq0}$ are $(a,b,\nu)$-TP.

(v) Let $\varphi(t)=\prod_{i=1}^m(1+x_it)$, $\omega=c=0$ and $a=1$.
Then we have
\begin{eqnarray*} P &=&\Lambda \Gamma(\varphi)\left[
\begin{array}{cccccc}
\nu & a &  &  &\\
\omega & \nu+b & 2a &\\
 & \omega+c & \nu+2b &3a &\\
  & &\omega+2c & \nu+3b &4a &\\
& && \ddots&\ddots & \ddots \\
\end{array}\right]\Lambda^{-1}\\
&=&
\left(\prod_{i=1}^m\Lambda\Gamma(1+x_it)\Lambda^{-1}\right)\Lambda\left[
\begin{array}{cccccc}
\nu & 1 &  &  &\\
& \nu+b & 2 &\\
 &  & \nu+2b &3 &\\
  & & & \nu+3b &4 &\\
& && &\ddots & \ddots \\
\end{array}\right]\Lambda^{-1}
\end{eqnarray*}
\begin{eqnarray*}
&=& \left(\prod_{i=1}^m\left[
  \begin{array}{ccccccc}
    1&& \\
    x_i&1&\\
    &2x_i&1&&\\
    &&3x_i&1&\\
 &&&\ddots&\ddots \\
  \end{array}
\right]\right)\left[
\begin{array}{cccccc}
\nu & 1 &  &  &\\
& \nu+b & 1 &\\
 &  & \nu+2b &1 &\\
  & & & \nu+3b &1 &\\
& && &\ddots & \ddots \\
\end{array}\right],
\end{eqnarray*}
which in terms of Proposition 7.2 and Proposition 8.2 (b) in
\cite{PSZ18} is exactly the production matrix for the $m$-branched
Stieltjes-type continued fraction
\begin{eqnarray*}
   \sum_{n\geq0}R_{n,0}t^n
   & = &
   \cfrac{1}
         {1 \,-\, \alpha_{m} t
            \prod\limits_{i_1=1}^{m}
                 \cfrac{1}
            {1 \,-\, \alpha_{m+i_1} t
               \prod\limits_{i_2=1}^{m}
               \cfrac{1}
            {1 \,-\, \alpha_{m+i_1+i_2} t
               \prod\limits_{i_3=1}^{m}
               \cfrac{1}{1 - \cdots}
            }
           }
         }
%
%   \cfrac{1}{1 - \cfrac{\alpha_{k+m} t}
%            {\prod\limits_{i_1=1}^{m}
%               \Biggl( \vspace*{-1cm}\displaystyle
%                      1 - \cfrac{{\alpha_{k+m+i_1} t}}
%                               {\displaystyle \prod\limits_{i_2=1}^{m}
%               \biggl( 1 - \frac{\alpha_{k+m+i_1+i_2} t}
%                               {1 - \cdots}
%               \biggr)
%                               }
%               \Biggr)
%            }}
\end{eqnarray*}
with coefficients $(\alpha_{i})_{i\geq
m}=(\nu,\underbrace{x_1,\ldots,x_m}_{m},\nu+b,\underbrace{2x_1,\ldots,2x_m}_{m},\nu+2b,\underbrace{3x_1,\ldots,3x_m}_{m},\ldots).$
\end{proof}
\begin{rem}
The $m$-branched Stieltjes-type continued fraction plays an
important role in the coefficientwise Hankel-total positivity. It
was shown that many combinatorial generating functions have
$m$-branched Stieltjes-type continued fraction expansions. We refer
the reader to \cite{PSZ18} for more details.
\end{rem}

\begin{prop}\label{prop+recp}
Let $A_n(q)$ be a polynomial of degree $n$. If $[A_{i+j}(q)]_{i,j}$
is $q$-TP$_r$, then so is $[A^*_{i+j}(q)]_{i,j}$.
\end{prop}
\begin{proof}
For $k\leq r$, let $\mathcal
{A}^{i_1,i_2,\ldots,i_k}_{j_1,j_2,\ldots,j_k}(q)$ (resp. $\mathcal
{B}^{i_1,i_2,\ldots,i_k}_{j_1,j_2,\ldots,j_k}(q)$) be the minor of
$[A_{i+j}(q)]_{i,j}$ (resp. $[A^*_{i+j}(q)]_{i,j}$) by taking its
rows ${i_1,i_2,\ldots,i_k}$ and columns ${j_1,j_2,\ldots,j_k}$. In
terms of the assumption that $[A_{i+j}(q)]_{i,j}$ is $q$-TP$_r$,
$\mathcal {A}^{i_1,i_2,\ldots,i_k}_{j_1,j_2,\ldots,j_k}(q)$ is a
nonnegative coefficients polynomial of degree $\leq
i_1+i_2+\cdots+i_k+j_1+j_2+\cdots+j_k$. Since
$A^*_{i+j}(q)=q^{i+j}A_{i+j}(1/q)$, taking out
$q^{i_1},q^{i_2},\ldots,q^{i_k}$ from rows and
$q^{j_1},q^{j_2},\ldots,q^{j_k}$ from columns of $\mathcal
{B}^{i_1,i_2,\ldots,i_k}_{j_1,j_2,\ldots,j_k}(q)$, respectively, we
get
\begin{eqnarray*}
\mathcal
{B}^{i_1,i_2,\ldots,i_k}_{j_1,j_2,\ldots,j_k}(q)&=&q^{i_1+i_2+\cdots+i_k+j_1+j_2+\cdots+j_k}\mathcal
{A}^{i_1,i_2,\ldots,i_k}_{j_1,j_2,\ldots,j_k}(1/q).
\end{eqnarray*}
In consequence, $\mathcal
{B}^{i_1,i_2,\ldots,i_k}_{j_1,j_2,\ldots,j_k}(q)$ is a nonnegative
coefficients polynomial. This completes the proof.
\end{proof}

By Theorem \ref{thm+Hakel} and Proposition \ref{prop+recp}, we get
the next result, which in particular implies Theorem
\ref{thm+Sheffer+polynomial}.

\begin{thm}\label{thm+Hankel+Rior}
Let $[E_{n,k}(f;\lambda)]_{n,k\geq0}=(\exp\left(\lambda
f(t)\right),f(t))$ and
$E_{n}(f;\lambda,q)=\sum_{k\geq0}E_{n,k}(f;\lambda)q^k$. If
$1/\bar{f}'(t)$ is a PF$_{r}$ function, then we have
\begin{itemize}
 \item [\rm (i)] the Hankel matrices
$[E_{i+j}(f;\lambda,q)]_{i,j\geq0}$ and
$[E^*_{i+j}(f;\lambda,q)]_{i,j\geq0}$ are $(\lambda,q)$-TP$_r$;
 \item [\rm (ii)]
the exponential Riordan array $[E_{n,k}(f;\lambda)]_{n,k\geq0}$ is
$\lambda$-TP$_r$;
 \item [\rm (iii)]
the Hankel matrix $[f_{i+j+1}]_{i,j\geq0}$ is TP$_r$;
 \item [\rm (iv)]
 the exponential Riordan array $(f'(x),f(x))$ is TP$_r$;
  \item [\rm (v)]
 the row-generating function of $(f'(x),f(x))$ forms a
$q$-Hankel-TP$_r$ sequence.
 \end{itemize}
\end{thm}

\begin{proof}
 For the exponential Riordan array
$(\exp\left(\lambda f(t)\right),f(t))$, by (\ref{generating
funtion+GR}), we have
$$\sum_{n\geq0}E_{n}(f;\lambda,q)\frac{t^n}{n!}=\exp\left(\lambda f(t)\right)\exp\left(qf(t)\right)=\exp\left((\lambda+q)f(t)\right).$$
In order to prove that the Hankel matrix
$[E_{i+j}(f;\lambda,q)]_{i,j\geq0}$ is $(\lambda,q)$-TP$_r$, it
suffices to prove that $[E_{i+j}(f;\lambda,0)]_{i,j\geq0}$ is
$\lambda$-TP$_r$ by taking $\lambda\rightarrow \lambda+q$.

For the exponential Riordan array $(\exp\left(\lambda
f(t)\right),f(t))$, by Proposition \ref{prop+TP+production}, there
exist two functions $Z(t)$ and $A(t)$ such that
$$Z(t)=\frac{\lambda}{\bar{f}'(t)},\quad
A(t)=\frac{1}{\bar{f}'(t)}.$$ It follows from (iv) of Theorem
\ref{thm+Hakel} that $[E_{n,k}(f;\lambda)]_{n,k\geq0}$ and
$[E_{i+j}(f;\lambda,0)]_{i,j\geq0}$ are $\lambda$-TP$_r$.
Furthermore, it follows from
$E^*_{n}(f;\lambda,q)=q^nE_{n}(f;\lambda,1/q)$ and Proposition
\ref{prop+recp} that $[E^*_{i+j}(f;\lambda,q)]_{i,j\geq0}$ is
$(\lambda,q)$-TP$_r$. This gives (i) and (ii).

By taking derivative in $t$ of
\begin{eqnarray}
\exp\left(\lambda
f(t)\right)&=&\sum_{n\geq0}E_n(f;\lambda,0)\frac{t^n}{n!},
\end{eqnarray}
we get
\begin{eqnarray}
\lambda f'(t)\exp\left(\lambda
f(t)\right)&=&\sum_{n\geq1}E_n(f;\lambda,0)\frac{t^{n-1}}{(n-1)!}.
\end{eqnarray}
That is \begin{eqnarray}\label{fucntion} f'(t)\exp\left(\lambda
f(t)\right)&=&\sum_{n\geq0}\frac{E_{n+1}(f;\lambda,0)}{\lambda}\frac{t^{n}}{n!}.
\end{eqnarray}

Since the Hankel matrix $[E_{i+j}(f;\lambda,0)]_{i,j\geq0}$ is
$\lambda$-TP$_r$ and the constant term of $E_n(f;\lambda,0)$ is
zero, $[\frac{E_{i+j+1}(f;\lambda,0)}{\lambda}]_{i,j\geq0}$ is
$\lambda$-TP$_r$. In particular, for $\lambda=0$, it reduces to that
$[f_{i+j+1}]_{i,j\geq0}$ is TP$_r$ by (\ref{fucntion}). In addition,
(\ref{fucntion}) implies
$$(f'(x),f(x))=[E_{n+1,k+1}(f;\lambda)]_{n,k\geq0}.$$ Then (iv) and (v) immediately follow from (i) and (ii).  This completes the
proof.
\end{proof}

\textbf{The proof of Theorem \ref{thm+Sheffer+convolution}:}

(i) and (ii) follow from Theorem \ref{thm+Sheffer+polynomial} and
Lemma \ref{lem+3-q-log-convex}.

For (iii), it follows from Theorem \ref{thm+ERA+TP}. Note a fact
that if a matrix is TP then so is its singless inverse, see
\cite{Pin10} for instance. Thus for (iv), it suffices to prove that
both $[L_{n,k}(f)]_{n,k}$ and $[\widetilde{L}_{n,k}(f)]_{n,k}$ are
TP, which follows from Theorem \ref{thm+Sheffer+triangle} (iii) and
(iv).

For (v), by (i) and Lemma \ref{lem+conv}, the convolution
$$z_n=\sum_{k=0}^nS_{n,k}(f,g)x_ky_{n-k}$$
preserves the SM property for $\lambda\geq0$. Note that we have
proved $\widetilde{L}_{n,k}(f)=k!S_{n,k}(f,1)$ in the proof of
Theorem \ref{thm+Sheffer+triangle}. Let $w_k=k!x_k$, which is a
Stieltjes moment sequence since the Hadamard product of two
Stieltjes moment sequences is still a Stieltjes moment sequence.
Then we get that the convolution
$$z_n=\sum_{k=0}^n\widetilde{L}_{n,k}(f)x_ky_{n-k}=\sum_{k=0}^nS_{n,k}(f,1)w_ky_{n-k}$$ preserves the SM property. In addition,
by Theorem \ref{thm+Hankel+Rior} (v) and Lemma \ref{lem+conv}, we
immediately get that the convolution
$$z_n=\sum_{k=0}^nS_{n+1,k+1}(f,1)x_ky_{n-k}$$ also preserves the SM
property.

Finally, for (vi), it follows from the proof of Theorem
\ref{thm+Hankel+Rior} that $$Z(t)=\frac{\lambda}{\bar{f}'(t)},\quad
A(t)=\frac{1}{\bar{f}'(t)}.$$ Then the first statement for $q=0$ of
(vi) is immediate from Theorem \ref{thm+Hakel} (v). By taking
$\lambda\rightarrow \lambda+q$, we get the first statement.
Moreover, the second statement is from the first one by taking
$q\rightarrow 1/q$ and $t\rightarrow qt$. This completes the proof.

\begin{thm}\label{thm+Hakel+gamma}
Define an exponential Riordan array
$$[R_{n,k}]_{n,k}=\left([ f(t)+1]^{\gamma}\exp(\lambda f(t)),f(t)\right).$$ If
$1/[\overline{f}'(t)(1+t)]$ is a PF$_r$ function, then we have
\begin{itemize}
 \item [\rm (i)]
the triangular matrix $[R_{n,k}]_{n,k}$ is
$(\gamma,\lambda)$-TP$_r$;
 \item [\rm (ii)]
 the Hankel matrices $[T_{i+j}(q)]_{i,j\geq0}$ and $[T^*_{i+j}(q)]_{i,j\geq0}$ are $(\gamma,q)$-TP$_{r}$, where
$\sum_{n\geq0}T_n(q)\frac{t^n}{n!}=[ f(t)+1]^{\gamma}\exp(qf(t));$
 \item [\rm (iii)]
the Hankel matrix $[f_{i+j+1}]_{i,j\geq0}$ is TP$_r$;
 \item [\rm (iv)]
 the Hankel matrices $[R_{i+j}(x)]_{i,j\geq0}$ and $[R^*_{i+j}(x)]_{i,j\geq0}$ are $(\gamma,\lambda,x)$-TP$_{r}$, where $R_n(x)=\sum_{k=0}^nR_{n,k}x^k$.
 \end{itemize}
\end{thm}

\begin{proof}
(i) Let $g(t)=[f(t)+1]^{\gamma}\exp(\lambda f(t))$. Then we derive
$$\frac{g'(t)}{g(t)}=\left(\frac{\gamma}{1+f(t)}+\lambda\right)f'(t).$$
For the exponential Riordan array $\left(g(t)),f(t)\right)$, by
Proposition \ref{prop+TP+production}, we derive
$$Z(t)=\frac{g'(\overline{f}(t))}{g(\overline{f}(t))}=\frac{\left(\gamma+\lambda+\lambda t\right)A(t)}{1+t}=\frac{\left(\gamma+\lambda+\lambda t\right)}{(1+t)\overline{f}'(t)}.$$
Let $\varphi(t)=\frac{1}{(1+t)\overline{f}'(t)}$. So for the
exponential Riordan array $\left(g(t)),f(t)\right)$ we get
$$Z(t)=(\gamma+\lambda+\lambda t)\varphi(t),\quad A(t)=(1+t)\varphi(t).$$
It follows from \cite[Proposition 3.3]{Zhu202} that the tridiagonal
matrix
$$\left[
\begin{array}{cccccc}
\gamma+\lambda & 1 &  &  &\\
\lambda & \gamma+\lambda+1 & 2 &\\
 & \lambda & \gamma+\lambda+2&3 &\\
  & &\lambda & \gamma+\lambda+3 &4 &\\
& && \ddots&\ddots & \ddots \\
\end{array}\right]$$
is $(\gamma,\lambda)$-TP. In consequence, by Theorem
\ref{thm+Hakel}, both $[R_{n,k}]_{n,k}$ and
$[T_{i+j}(\lambda)]_{i,j}$ are $(\gamma,\lambda)$-TP$_r$. Obviously,
Furthermore, by Proposition \ref{prop+recp}, we have
$[T^*_{i+j}(\lambda)]_{i,j}$ is $(\gamma,\lambda)$-TP$_r$. Thus, we
show that (i) and (ii) hold.

(iii) By the hypothesis that $1/[\overline{f}'(t)(1+t)]$ is a PF$_r$
function, $1/\overline{f}'(t)$ is a PF$_r$ function. Then taking
$\gamma=0$, we have the Hankel matrix $[f_{i+j+1}]_{i,j\geq0}$ is
TP$_r$ by Theorem \ref{thm+Hankel+Rior} (iii).

 (iv) For the row-generating function $R_n(x)$, we have
\begin{equation*}
\sum_{n\geq0}R_n(x)\frac{t^n}{n!}=[f(t)+1]^{\gamma}\exp((\lambda+x)f(t)).
\end{equation*}
Thus we get $R_n(x)=T_n(\lambda+x)$ for $n\geq0$. In consequence,
(iv) is immediate from (ii).
\end{proof}

Similar to the proof of Theorem \ref{thm+Sheffer+convolution}, by
Theorem \ref{thm+Hakel+gamma}, Lemma \ref{lem+3-q-log-convex}, and
Lemma \ref{lem+conv}, we get Theorem
\ref{thm+Sheffer+convolution+2}.

\section{Applications}

Let $\left\langle
  \begin{array}{ccccc}
    n \\
   k\\
  \end{array}
\right\rangle$ denote the Eulerian number, which counts the number
of $n$-permutations with exactly $k-1$ excedances. It is well-known
that it satisfies the recurrence relation
\begin{equation*}\label{rr-an}
\left\langle
  \begin{array}{ccccc}
    n \\
   k\\
  \end{array}
\right\rangle=k\left\langle
  \begin{array}{ccccc}
    n-1 \\
   k\\
  \end{array}
\right\rangle+(n-k+1)\left\langle
  \begin{array}{ccccc}
    n-1 \\
   k-1\\
  \end{array}
\right\rangle
\end{equation*}
with initial conditions $\left\langle
  \begin{array}{ccccc}
    0 \\
   0\\
  \end{array}
\right\rangle=1$ and $\left\langle
  \begin{array}{ccccc}
    0 \\
   k\\
  \end{array}
\right\rangle=0$ for $k\geq1$ or $k<0$. The triangular array
$\left[\left\langle
  \begin{array}{ccccc}
    n \\
   k\\
  \end{array}
\right\rangle\right]_{n,k\geq0}$ is called the Eulerian triangle.
Brenti proposed the next conjecture, which is still open.
\begin{conj} \emph{\cite[Conjecture 6.10]{Bre96}}
The Eulerian triangle is TP.
\end{conj}
Generally, let $[E_{n,k}]_{n,k\geq0}$ be an array satisfying the
recurrence relation:
\begin{equation}\label{recurrence+two term+triangle}
E_{n,k}=[a_0n+a_1k+a_2]E_{n-1,k}+[b_0n+b_1k+b_2]E_{n-1,k-1}
\end{equation}
with $E_{n,k}=0$ unless $0\le k\le n$ and $E_{0,0}=1$. As we know
that many classical enumerations satisfy the recurrence relation
(\ref{recurrence+two term+triangle}). In addition, some special
cases of $[E_{n,k}]_{n,k\geq0}$ are totally positive, for example,
for $a_0=b_0=0$ or $a_1=b_1=0$, see Brenti \cite{Bre95}. In this
section, we will present more results for total positivity of
triangular arrays.

\subsection{A generalization of Bessel numbers of the second kind}

For $\{n,k\}\subseteq \mathbb{N}$, the {\it Bessel number of the
second kind} $B_{n,k}$ is defined to be the number of partitions of
$[n]:=\{1,2,3,\ldots,n\}$ into $k$ nonempty blocks of size at most
$2$. It satisfies the recurrence relation
\begin{eqnarray}
B_{n,k}&=&B_{n-1,k-1}+(n-1)B_{n-2,k-1}
\end{eqnarray}
with the initial condition $B_{0,k}=\delta_{0,k}$. We can get some
properties of $B_{n,k}$ from the following more generalized
recurrence relation for $a=1$, $b=\frac{1}{2}$ and $c=0$.

\begin{prop}
Let $\{a,b,c\}\subseteq \mathbb{R}^{\geq0}$. Assume that a
triangular array $[\mathcal {B}_{n,k}]_{n,k\geq0}$ satisfies the
following recurrence relation
\begin{eqnarray}\label{rec+B}
\mathcal {B}_{n,k}&=&a\mathcal {B}_{n-1,k-1}+2b(n-1)\mathcal
{B}_{n-2,k-1}+3c(n-1)(n-2)\mathcal {B}_{n-3,k-1}
\end{eqnarray}
for $n,k\geq1$, where $\mathcal {B}_{1,1}=1$ and $\mathcal
{B}_{0,k}=\delta_{0,k}$. Then we have the following results.
\begin{itemize}
 \item [\rm (i)]
 An explicit formula of $\mathcal {B}_{n,k}$ can be written as
\begin{eqnarray}
\mathcal
{B}_{n,k}&=&\frac{n!}{k!}\sum_{i=0}^ka^{k-i}c^{n-k-i}b^{2i-n+k}\binom{k}{i}\binom{i}{n-k-i}
\end{eqnarray}
for $n,k\geq1$.
\item [\rm (ii)]
 The lower-triangular matrix $[\mathcal
{B}_{n,k}]_{n,k\geq0}$ is TP$_r$ for $b^2\geq 4ac \cos^2
\frac{\pi}{r+1}$.
\item [\rm (iii)]
 The lower-triangular matrix $[\mathcal
{B}_{n,k}]_{n,k\geq0}$ is TP for $b^2\geq 4ac$.
 \end{itemize}
\end{prop}
\begin{proof}
Let $f_k(t)=\sum_{n\geq k}\mathcal{B}_{n,k}\frac{t^n}{n!}$. After we
multiply $\frac{t^{n-1}}{(n-1)!}$ for both sides of (\ref{rec+B})
and sum on $n$, we get
$$f_k'(t)=(a+2bt+3ct^2)f_{k-1}(t)$$ with $f_0(t)=1$ and $f_k(0)=0$ for
$k\geq1$. This implies that $$f_k(t)=\frac{(at+bt^2+ct^3)^k}{k!}$$
for $k\geq1$. Thus the triangular array $[\mathcal
{B}_{n,k}]_{n,k\geq0}$ is the exponential Riordan array
$$(1,at+bt^2+ct^3).$$ Then we immediately get
\begin{eqnarray*}
\mathcal {B}_{n,k}&=&\frac{n!}{k!}[t^n]f^k(t)\\
&=&\frac{n!}{k!}[t^{n-k}](a+bt+ct^2)^k\\
&=&\frac{n!}{k!}[t^{n-k}]\sum_{i=0}^k\binom{k}{i}(bt+ct^2)^ia^{k-i}\\
&=&\frac{n!}{k!}\sum_{i=0}^k\binom{k}{i}[t^{n-k-i}](b+ct)^ia^{k-i}\\
&=&\frac{n!}{k!}\sum_{i=0}^ka^{k-i}c^{n-k-i}b^{2i-n+k}\binom{k}{i}\binom{i}{n-k-i},
\end{eqnarray*}
which is (i).

 By Theorem \ref{thm+Sheffer+triangle}, for (ii) and
(iii), it suffices to prove that $a,b,c$ forms a PF$_r$ sequence for
$b^2\geq 4ac \cos \frac{\pi}{r+1}$. This is true because its
Toeplitz matrix is TP$_r$ for $b^2\geq 4ac \cos^2 \frac{\pi}{r+1}$,
see \cite{KV06}.
\end{proof}

\begin{rem}
A combinatorial interpretation of $\mathcal {B}_{n,k}$ can be given
as follows: in any ordered partition of $[n]$, if a block only
contains $n$ then we give it a weight $a$; if a block of size $2$
contains $n$ then we give it a weight $b$; if a block of size $3$
contains $n$ then we give it a weight $c$; for other block with
weight $1$. Then
$$\mathcal
{B}_{n,k}=\sum_{\{\mathscr{B}_1,\mathscr{B}_2,\mathscr{B}_3,\ldots\},1\leq|\mathscr{B}_i|\leq3}\prod_i
w(\mathscr{B}_i),$$ where $\{\mathscr{B}_1,\mathscr{B}_2,\ldots
\mathscr{B}_k\}$ is a partition of $[n]$ into $k$ nonempty ordered
blocks of size at most $3$.
\end{rem}

\subsection{A generalization of Bessel numbers of the first kind}

The famous Bessel polynomial introduced by Krall and Frink
\cite{KF45} has the following formula
$$y_n(q)=\sum_{k=0}^n\frac{(n+k)!}{(n-k)!k!}\left(\frac{q}{2}\right)^k,$$
which satisfies a second-order differential equation
$$q^2y''_n+(2q+2)y'_n-n(n+1)y_n=0$$
and the recurrence relation
$$y_{n+1}(q)=(2n+1)qy_n(q)+y_{n-1}(q)$$
with initial conditions $y_0(q)=1$ and $y_1(q)=1+q$. It is
well-known that Bessel polynomials form an orthogonal sequence of
polynomials. See \cite{Car57,Gro51} for more properties of Bessel
polynomials. If we write
$$y_{n-1}=\sum_{k\geq0}b_{n,k}q^{n-k},$$ then $b_{n,k}$ is called
{\it the signless Bessel number} of the first kind. Denote the
reverse Bessel polynomial by $Y_n(q)=\sum_{k\geq0}b_{n,k}q^k$.

The reverse Bessel polynomial and the signless Bessel number satisfy
$$\exp(q(1-\sqrt{1-2t}))=\sum_{n\geq0} Y_n(q) \frac{t^n}{n!}=\sum_{n\geq0}
\sum_{k=0}^nb_{n,k}q^k \frac{t^n}{n}.$$ Note that $[b_{n,k}]_{n,k}$
satisfies the following recurrence relation:
\begin{eqnarray}\label{rec+Bes}
b_{n,k}=(2n-k-2)b_{n-1,k}+b_{n-1,k-1}
\end{eqnarray}
with $b_{0,0}=1$. Generalizing the recurrence relation
(\ref{rec+Bes}), we define a {\it generalized Bessel triangle}
$[\mathcal {T}_{n,k}]_{n,k\geq0}$ and get some general results as
follows, which in particular implies properties of the signless
Bessel triangle $\left[b_{n,k}\right]_{n,k\geq0}$ and the reverse
Bessel polynomial $Y_n(q)$ for $a=2$, $b=c=1$ and $\lambda=d=0$.

\begin{prop}\label{prop+two+term+triangle-}
Let $a\in \mathbb{N}^{+}$, $ad\in \mathbb{N}$ and $c\in
\mathbb{R}^{+}$. Assume that a generalized Bessel triangle
$[\mathcal {T}_{n,k}]_{n,k\geq0}$ satisfies the following recurrence
relation
\begin{eqnarray}\label{rec+four}
\mathcal {T}_{n,k}&=&c \mathcal
{T}_{n-1,k-1}+\left[ab(n-1)-bk+abd+c\lambda\right]\mathcal
{T}_{n-1,k}-b\lambda (k+1)\mathcal {T}_{n-1,k+1}
\end{eqnarray}
for $n,k\geq1$, where $\mathcal {T}_{0,0}=1$. Let $\mathcal
{T}_n(q)=\sum_{k\geq0}\mathcal {T}_{n,k}q^k$. Then we have the
following results.
\begin{itemize}
 \item [\rm (i)]
 An explicit formula of $\mathcal {T}_{n,k}$ can be written as
\begin{eqnarray}
\mathcal
{T}_{n,k}&=&\frac{a^n}{k!}\sum_{j\geq0}\sum_{i=0}^{k+j}\frac{(-1)^{n-i}b^{n-k-j}c^{k+j}\lambda^j}{j!}\binom{k+j}{i}\left(\frac{i-ad}{a}\right)_n
\end{eqnarray}
for $n,k\geq1$.
 \item [\rm (ii)]
 The lower-triangular matrix $[\mathcal {T}_{n,k}]_{n,k\geq0}$ is $(b,\lambda)$-TP.
 \item [\rm (iii)]
For $0\leq ad\leq a-1$, $(\mathcal {T}_{n}(q))_{n\geq0}$ is
$(b,\lambda,q)$-SM and $3$-$(b,\lambda,q)$-log-convex. In
particular, $(\mathcal {T}_{n,0})_{n\geq0}$ is $(b,\lambda)$-SM and
$3$-$(b,\lambda)$-log-convex.
\item [\rm (iv)]
The sequence $(f_{n})_{n\geq1}$ is SM, where
$\sum_{n\geq1}f_n\frac{t^n}{n!}=\frac{c}{b}\left[1-(1-abt)^{\frac{1}{a}}\right]$
and $b>0$.
\item [\rm (v)]
If $0\leq ad\leq a-1$ and $\{b,\lambda\}\subseteq
\mathbb{R}^{\geq0}$, then $z_n=\sum_{k\geq0}\mathcal
{T}_{n,k}x_ky_{n-k}$ preserves Stieltjes moment property of
sequences.
 \end{itemize}
\end{prop}

\begin{proof}

Let $f_k(t)=\sum_{n\geq k}\mathcal {T}_{n,k}\frac{t^n}{n!}$. It
follows from the recurrence relation (\ref{rec+four}) that
$$f'_k(t)=abtf'_k(t)+[b(ad-k)+c\lambda] f_k(t)+c f_{k-1}(t)-b\lambda
(k+1)f_{k+1}$$ for $k\geq1$. This implies
\begin{eqnarray}\label{eq+fkk}
b\lambda
(k+1)f_{k+1}+(1-abt)f'_k(t)-cf_{k-1}(t)+[b(k-ad)-c\lambda]f_k(t)=0
\end{eqnarray}
for $k\geq1$. It is not hard to check that
$$f_k(t)=\frac{1}{k!}(1-abt)^{-d}e^{\frac{c\lambda}{b}\left[1-(1-abt)^{\frac{1}{a}}\right]}\left(\frac{c}{b}\right)^k\left[1-(1-abt)^{\frac{1}{a}}\right]^k$$
is the solution of equation (\ref{eq+fkk}). Let
$$g(t)=(1-abt)^{-d},\quad
f(t)=\frac{c}{b}\left[1-(1-abt)^{\frac{1}{a}}\right].$$ Then we have
$[\mathcal {T}_{n,k}]_{n,k}$ is the exponential Riordan array
$(g(t)e^{\lambda f(t)},f(t))$. So by taking the coefficients of
$t^n$ in $g(t)e^{\lambda f(t)}f(t)^k$, we immediately get
\begin{eqnarray*}
\mathcal
{T}_{n,k}&=&\frac{n!}{k!}[t^n]g(t)e^{\lambda f(t)}f(t)^k\\
&=&\frac{n!}{k!}[t^n]g(t)\sum_{j\geq0}\frac{{\lambda}^jf^{j+k}(t)}{j!}\\
&=&\left(\frac{c}{b}\right)^k\frac{n!}{k!}[t^n](1-abt)^{-d}\sum_{j\geq0}\frac{{\left(\frac{c\lambda}{b}\right)}^j}{j!}\sum_{i=0}^{j+k}(-1)^i\binom{j+k}{i}(1-abt)^{\frac{i}{a}}\\
&=&\left(\frac{c}{b}\right)^k\frac{n!}{k!}[t^n]\sum_{j\geq0}\sum_{i=0}^{k+j}(-1)^i\frac{{\left(\frac{c\lambda}{b}\right)}^j}{j!}\binom{k+j}{i}(1-abt)^{\frac{i-ad}{a}}\\
&=&\left(\frac{c}{b}\right)^k\frac{n!}{k!}\sum_{j\geq0}\sum_{i=0}^{k+j}(-1)^i\frac{{\left(\frac{c\lambda}{b}\right)}^j}{j!}\binom{k+j}{i}(-ab)^n\frac{\left(\frac{i-ad}{a}\right)_n}{n!}\\
&=&\frac{a^n}{k!}\sum_{j\geq0}\sum_{i=0}^{k+j}\frac{(-1)^{n-i}b^{n-k-j}c^{k+j}\lambda^j}{j!}\binom{k+j}{i}\left(\frac{i-ad}{a}\right)_n,
\end{eqnarray*}
which is (i).

 Note that the compositional
inverse of $f(t)$ is
$$\bar{f}(t)=\frac{1-\left(1-\frac{b}{c}t\right)^{a}}{ab}.$$ Clearly, both
$g(t)$ and $1/\bar{f}'(t)=\frac{c}{(1-\frac{b}{c}t)^{a-1}}$ are
P\'olya frequency functions for $a\in \mathbb{N}^{+}$, $d\in
\mathbb{N}$ and $c\in \mathbb{R}^{+}$. Then (ii) follows from
Theorem \ref{thm+ERA+TP}.

For $d=0$, by Theorem \ref{thm+Hankel+Rior} (i) and (iii), we obtain
that $(\mathcal {T}_{n}(q))_{n\geq0}$ is $(\lambda,q)$-SM and
$(f_{n})_{n\geq1}$ is SM. Furthermore, $(\mathcal
{T}_{n}(q))_{n\geq0}$ is $3$-$(\lambda,q)$-log-convex and
$z_n=\sum_{k\geq0}\mathcal {T}_{n,k}x_ky_{n-k}$ preserves Stieltjes
moment property of sequences by Lemmas \ref{lem+3-q-log-convex} and
\ref{lem+conv}, respectively.

For the general case $d\neq0$, the corresponding total positivity in
(iii) can be proved from the following production matrix. In
addition, for the exponential Riordan array $(g(t)e^{\lambda
f(t)},f(t))$, by (\ref{generating funtion+GR}), we have
\begin{eqnarray}\label{REG-}
\sum_{n\geq0}\mathcal {T}_{n}(q)\frac{t^n}{n!}=g(t)\exp\left(\lambda
f(t)\right)\exp\left(qf(t)\right)=g(t)\exp\left((\lambda+q)f(t)\right).
\end{eqnarray}
Thus, in order to prove that the Hankel matrix $[\mathcal
{T}_{i+j}(q)]_{i,j\geq0}$ is $(b,\lambda,q)$-TP, it suffices to
prove that $[\mathcal {T}_{i+j}(0)]_{i,j\geq0}$ is $(b,\lambda)$-TP,
i.e., $[\mathcal {T}_{i+j,0}]_{i,j\geq0}$ is $(b,\lambda)$-TP.

For $Z(t)$ and $A(z)$ of the exponential Riordan array
$(g(t)e^{\lambda f(t)},f(t))$, we derive
$$Z(t)=\frac{abd}{\left(1-\frac{b}{c}t\right)^{a}}+ \frac{\lambda c}{\left(1-\frac{b}{c}t\right)^{a-1}},
A(t)=\frac{c}{\left(1-\frac{b}{c}t\right)^{a-1}}.$$ We assume
$p=ad\in[0,a-1]$. Then let
$\varphi(t)=\frac{1}{\left(1-\frac{b}{c}t\right)^{a-p-1}}$ and
$\phi(t)=\frac{c}{\left(1-\frac{b}{c}t\right)^{p}}$. We deduce the
production matrix $P$ of the exponential Riordan array
$(g(t)e^{\lambda f(t)},f(t))$ as follows:
\begin{eqnarray}\label{PPPP}
P&=&\Lambda \Gamma(Z)\Lambda^{-1}+\Lambda \Gamma(A)\Theta\Lambda^{-1}\nonumber\\
&=&\Lambda \Gamma(\varphi)\left(\Gamma\left(\frac{abd}{\left(1-\frac{b}{c}t\right)^{p+1}}+ \lambda \phi(t)\right)+ \Gamma(\phi(t))\Theta\right)\Lambda^{-1}\nonumber\\
&=&\Lambda
\Gamma(\varphi)\left(\Gamma\left(\frac{pb}{\left(1-\frac{b}{c}t\right)^{p+1}}+
\lambda \phi(t)\right)+ \Gamma(\phi(t))\left[
  \begin{array}{cccccccc}
    0&1&& \\
    &0&2&\\
   & &0&3&&\\
   & &&0&4&\\
 &&&&\ddots&\ddots \\
  \end{array}
\right]\right)\Lambda^{-1}\nonumber\\
&=&\Lambda
\Gamma(\varphi)\left[c(i+1)\binom{p+i-j}{i-j+1}\left(\frac{b}{c}\right)^{i-j+1}+\lambda
c\binom{p-1+i-j}{i-j}\left(\frac{b}{c}\right)^{i-j}\right]_{i,j}\Lambda^{-1}\nonumber\\
 &=&\Lambda \Gamma(\varphi)\left[
\begin{array}{cccccc}
\lambda c & c &  &  &\\
 & \lambda c & 2c &\\
 &  & \lambda c &3c &\\
  & & & \lambda c &4c &\\
& && &\ddots & \ddots \\
\end{array}\right]\Gamma(\phi)\Lambda^{-1},
\end{eqnarray}
which is $(b,\lambda)$-TP. It follows from Theorem \ref{thm+Hakel}
(i) that we get the $(b,\lambda)$-total-positivity of $[\mathcal
{T}_{n,k}]_{n\geq0}$ in (ii), $(b,\lambda,q)$-Stieltjes moment
property of $(\mathcal {T}_{n}(q))_{n\geq0}$, and
$(b,\lambda)$-Stieltjes moment property of $(\mathcal
{T}_{n,0})_{n\geq0}$ in (iii). Furthermore, the remaining results in
(iii) and (v) are immediate by Lemma \ref{lem+3-q-log-convex} and
Lemma \ref{lem+conv}, respectively.
\end{proof}

Let $[\mathcal {T}_{n,k}]_{n,k}$ be defined by (\ref{rec+four}). For
the reciprocal generalized Bessel polynomial $\mathcal
{T}_n^{*}(q)=q^n\mathcal {T}_n(1/q)$ and the reciprocal generalized
Bessel numbers $\mathcal {T}^*_{n,k}=\mathcal {T}_{n,n-k}$, by
Proposition \ref{prop+two+term+triangle-}, we immediately get the
following reciprocal result.

\begin{prop}
Let $a\in \mathbb{N}^{+}$, $ad\in \mathbb{N}$ and $c\in
\mathbb{R}^{+}$. If $0\leq ad\leq a-1$, then we have
\begin{itemize}
 \item [\rm (i)]
the reciprocal generalized Bessel triangle $[\mathcal
{T}^*_{n,k}]_{n,k\geq0}$ satisfies the following recurrence
\begin{eqnarray*}
\mathcal {T}^*_{n,k}&=&c \mathcal
{T}^*_{n-1,k}+\left[ab(n-1)-b(n-k)+abd+c\lambda\right]\mathcal
{T}^*_{n-1,k-1}-b\lambda (n-k+1)\mathcal {T}^*_{n-1,k-2}
\end{eqnarray*}
for $n,k\geq1$, where $\mathcal {T}^*_{0,0}=1$;
 \item [\rm (ii)]
 the exponential generating function of $\mathcal
{T}_n^{*}(q)$ can be written as
 \begin{eqnarray}
\sum_{n\geq0}\mathcal
{T}^*_{n}(q)\frac{t^n}{n!}=g(qt)\exp\left((q\lambda+1)f(qt)/q\right);
\end{eqnarray}
 \item [\rm (iii)]
$(\mathcal {T}^*_{n}(q))_{n\geq0}$ is $(b,\lambda,q)$-SM and
$3$-$(b,\lambda,q)$-log-convex. In particular, $(\mathcal
{T}^*_{n,n})_{n\geq0}$ is $(b,\lambda)$-SM and
$3$-$(b,\lambda)$-log-convex.
\item [\rm (iv)]
the convolution $z_n=\sum_{k\geq0}\mathcal {T}^*_{n,k}x_ky_{n-k}$
preserves Stieltjes moment property of sequences for
$\{b,\lambda\}\subseteq \mathbb{R}^{\geq0}$.
 \end{itemize}
\end{prop}

In \cite{Cal09}, Callan proved
$$\sum_{k=0}^nk!\binom{2n-k-1}{k-1}(2n-2k-1)!!=(2n-1)!!,$$
which in fact counts different combinatorial structures, such as
increasing ordered trees of $n$ edges by outdegree $k$ of the root.
Let $$H_{n,k}=k!\binom{2n-k-1}{k-1}(2n-2k-1)!!,$$ which is the sum
of the weights of all vertices labeled $k$ at depth $n$ in the
Catalan tree for $1\leq k \leq n+1$ and $n\geq0$, see
\cite[A102625]{Slo}. In addition, $H_{n,k}$ satisfies the recurrence
relation
\begin{eqnarray}\label{rel+Call}
H_{n,k}&=&(2n-k-2)H_{n-1,k}+kH_{n-1,k-1}
\end{eqnarray}
with initial condition $H_{0,0}=1$. In fact, $H_{n,k}$ is closely
related to the signless Bessel number $b_{n,k}$ by
$H_{n,k}=b_{n,k}k!$. Generalizing this relation (\ref{rel+Call}), we
consider the following triangle $[\widetilde{\mathcal
{T}}_{n,k}]_{n,k\geq0}$ related to the generalized Bessel triangle
in Proposition (\ref{prop+two+term+triangle-})  by
$\widetilde{\mathcal {T}}_{n,k}=\mathcal {T}_{n,k}k!.$ Thus, by
Theorem \ref{thm+Sheffer+convolution} and Proposition
\ref{prop+two+term+triangle-}, we have the next result.

\begin{prop}\label{prop+two+term+triangle-k}
Let $a\in \mathbb{N}^{+}$, $ad\in \mathbb{N}$ and
$c\in\mathbb{R}^{+}$. Assume that a triangular array
$[\widetilde{\mathcal {T}}_{n,k}]_{n,k\geq0}$ satisfies the
following recurrence relation
\begin{eqnarray}
\widetilde{\mathcal {T}}_{n,k}&=&ck \widetilde{\mathcal
{T}}_{n-1,k-1}+\left[ab(n-1)-bk+abd+c\lambda\right]\widetilde{\mathcal
{T}}_{n-1,k}-b\lambda \widetilde{\mathcal {T}}_{n-1,k+1}
\end{eqnarray}
for $n,k\geq1$, where $\widetilde{\mathcal {T}}_{0,0}=1$. Let
$\widetilde{\mathcal {T}}_n(q)=\sum_{k\geq0}\widetilde{\mathcal
{T}}_{n,k}q^k$. Then we have the following results.
\begin{itemize}
 \item [\rm (i)]
 An explicit formula of $\widetilde{\mathcal {T}}_{n,k}$ can be written as
\begin{eqnarray*}
\widetilde{\mathcal
{T}}_{n,k}&=&a^n\sum_{j\geq0}\sum_{i=0}^{k+j}\frac{(-1)^{n-i}b^{n-k-j}c^{k+j}\lambda^j}{j!}\binom{k+j}{i}\left(\frac{i-ad}{a}\right)_n
\end{eqnarray*}
for $n,k\geq1$.
 \item [\rm (ii)]
 The lower-triangular matrix $[\widetilde{\mathcal
{T}}_{n,k}]_{n,k\geq0}$ is $(b,\lambda)$-TP.
 \item [\rm (iii)]
For $d=0$, the exponential generating function is
\begin{eqnarray*}
1+\sum_{n\geq1}\widetilde{\mathcal
{T}}_n(q)\frac{t^n}{n!}=\frac{1}{1-\frac{c}{b}\left[1-(1-abt)^{\frac{1}{a}}\right](q+\lambda)}.
\end{eqnarray*}
 \item [\rm (iv)]
If $0\leq ad\leq a-1$ and $\{b,\lambda\}\subseteq
\mathbb{R}^{\geq0}$, then the convolution
$z_n=\sum_{k\geq0}\widetilde{\mathcal {T}}_{n,k}x_ky_{n-k}$
preserves Stieltjes moment property of sequences.
 \end{itemize}
\end{prop}

\subsection{A generalization of Lah numbers}

It is well known that the exponential partial Bell polynomial
$\textbf{B}_{n,k}(x_1,x_2,\ldots,x_{n-k+1})$ is defined by the
series expansion:
\begin{equation}\label{Partial Bell+seq}
\sum_{n\geq0}\textbf{B}_n(\textbf{x};q)\frac{t^n}{n!}=\exp\left(q\sum_{j\geq1}x_j\frac{t^j}{j!}\right),
\end{equation}
where \begin{equation}\label{Shef+seq}
\textbf{B}_n(\textbf{x};q)=\sum_{k=1}^n\textbf{B}_{n,k}(x_1,\ldots,x_{n-k+1})q^k\quad
\text{for} \quad n\geq1
\end{equation}
and $\textbf{B}_0(\textbf{x};q)=1$. For $q=1$,
$\textbf{B}_n(\textbf{x};1)$ is called {\it the exponential complete
Bell polynomial}. See Comtet \cite[p.133-134]{Com74}.

Note that the well-known signless Lah number
$$\mathcal {L}_{n,k}=\textbf{B}_{n,k}(1!,2!,3!,\ldots)=\binom{n-1}{k-1}\frac{n!}{k!},$$
which satisfies
$$\exp\left(\frac{qt}{1-t}\right)=\sum_{n\geq0} \mathcal {L}_n(q) \frac{t^n}{n!}=\sum_{n\geq0}
\sum_{k=0}^n\mathcal {L}_{n,k}q^k \frac{t^n}{n},$$ see
\cite[p.133-134]{Com74}. It counts the number of partitions of $[n]$
into $k$ lists, where a list means an ordered subset
\cite[A008297]{Slo}. In addition, the row-generating function
$$\mathcal {L}_n(q)=\sum_{k=0}^n\mathcal {L}_{n,k}q^k$$
is called {\it the Lah polynomial}. Notice that $[\mathcal
{L}_{n,k}]_{n,k}$ satisfies the following recurrence relation:
\begin{eqnarray}\label{rec+Lah}
\mathcal {L}_{n,k}=(n+k-1)\mathcal {L}_{n-1,k}+\mathcal
{L}_{n-1,k-1}
\end{eqnarray}
for $n,k\geq1$ with $\mathcal {L}_{0,0}=1$.
 Generalizing the recurrence relation (\ref{rec+Lah}), we give the
 following result for a generalized Lah triangle $[T_{n,k}]_{n,k\geq0}$, which reduces to those of the signless Lah triangle
$\left[\mathcal {L}_{n,k}\right]_{n,k\geq0}$ for $a=b=c=1$ and
$d=\lambda=0$.

\begin{prop}\label{prop+two+term+triangle}
Let $a\in \mathbb{N}$. Assume that a generalized Lah triangle
$[T_{n,k}]_{n,k\geq0}$ satisfies the following recurrence relation
\begin{eqnarray}\label{rec+four+}
T_{n,k}&=&c\,
T_{n-1,k-1}+\left[ab(n-1)+bk+abd+c\lambda\right]T_{n-1,k}+b\lambda(k+1)
T_{n-1,k+1}
\end{eqnarray}
for $n,k\geq1$, where $T_{0,0}=1$. Let
$T_n(q)=\sum_{k\geq0}T_{n,k}q^k$. Then we have the following
results.
\begin{itemize}
\item [\rm (i)]
 An explicit formula of $T_{n,k}$ can be given by
\begin{eqnarray*}
T_{n,k}&=&\frac{a^n}{k!}\sum_{j\geq0}\sum_{i=0}^{k+j}\frac{(-1)^{n+k+j-i}b^{n-k-j}c^{k+j}\lambda^j}{j!}\binom{k+j}{i}\left(-\frac{i+ad}{a}\right)_n
\end{eqnarray*}
for $n,k\geq1$.
 \item [\rm (ii)]
For $d\in \mathbb{N}$ and $c\in\mathbb{R}^{+}$, the lower-triangular
matrix $[T_{n,k}]_{n\geq0}$ is $(b,\lambda)$-TP.
 \item [\rm (iii)]
For $c=b$, the lower-triangular matrix $[T_{n,k}]_{n\geq0}$ is
$(b,d,\lambda)$-TP.
\item [\rm (iv)]
For $c=b$, the sequence $(T_{n}(q))_{n\geq0}$ is
$(b,d,\lambda,q)$-SM and $3$-$(b,d,\lambda,q)$-log-convex. In
particular, $(T_{n,0})_{n\geq0}$ is $(b,d,\lambda)$-SM and
$3$-$(b,d,\lambda)$-log-convex.
\item [\rm (v)]
For $b>0$, the sequence $(f_{n})_{n\geq1}$ is SM, where
$\sum_{n\geq1}f_n\frac{t^n}{n!}=\frac{(1-abt)^{-\frac{1}{a}}-1}{b}$.
\item [\rm (vi)]
Let $c=b$. For $\{b,d,\lambda\}\subseteq \mathbb{R}^{\geq0}$, the
convolution $z_n=\sum_{k\geq0}T_{n,k}x_ky_{n-k}$ preserves Stieltjes
moment property of sequences.
\item [\rm (vii)]
Let $a+1=m$ and $d=0$. We have the $m$-branched Stieltjes-type
continued fraction expression
\begin{eqnarray*}
   \sum_{n\geq0}T_{n}(q)t^n
   & = &
   \cfrac{1}
         {1 \,-\, \alpha_{m} t
            \prod\limits_{i_1=1}^{m}
                 \cfrac{1}
            {1 \,-\, \alpha_{m+i_1} t
               \prod\limits_{i_2=1}^{m}
               \cfrac{1}
            {1 \,-\, \alpha_{m+i_1+i_2} t
               \prod\limits_{i_3=1}^{m}
               \cfrac{1}{1 - \cdots}
            }
           }
         }
%
%   \cfrac{1}{1 - \cfrac{\alpha_{k+m} t}
%            {\prod\limits_{i_1=1}^{m}
%               \Biggl( \vspace*{-1cm}\displaystyle
%                      1 - \cfrac{{\alpha_{k+m+i_1} t}}
%                               {\displaystyle \prod\limits_{i_2=1}^{m}
%               \biggl( 1 - \frac{\alpha_{k+m+i_1+i_2} t}
%                               {1 - \cdots}
%               \biggr)
%                               }
%               \Biggr)
%            }}
\end{eqnarray*}
with coefficients
$(\alpha_{i})_{i\geq{m}}=(c(q+\lambda),\underbrace{b,\ldots,b}_{m},c(q+\lambda),\underbrace{2b,\ldots,2b}_{m},c(q+\lambda),\underbrace{3b,\ldots,3b}_{m},\ldots).$
In particular, for $q=0$, we have the $m$-branched Stieltjes-type
continued fraction expression
\begin{eqnarray*}
   \sum_{n\geq0}T_{n,0}t^n
   & = &
   \cfrac{1}
         {1 \,-\, \alpha_{m} t
            \prod\limits_{i_1=1}^{m}
                 \cfrac{1}
            {1 \,-\, \alpha_{m+i_1} t
               \prod\limits_{i_2=1}^{m}
               \cfrac{1}
            {1 \,-\, \alpha_{m+i_1+i_2} t
               \prod\limits_{i_3=1}^{m}
               \cfrac{1}{1 - \cdots}
            }
           }
         }
%
%   \cfrac{1}{1 - \cfrac{\alpha_{k+m} t}
%            {\prod\limits_{i_1=1}^{m}
%               \Biggl( \vspace*{-1cm}\displaystyle
%                      1 - \cfrac{{\alpha_{k+m+i_1} t}}
%                               {\displaystyle \prod\limits_{i_2=1}^{m}
%               \biggl( 1 - \frac{\alpha_{k+m+i_1+i_2} t}
%                               {1 - \cdots}
%               \biggr)
%                               }
%               \Biggr)
%            }}
\end{eqnarray*}
with coefficients
$(\alpha_{i})_{i\geq{m}}=(c\lambda,\underbrace{b,\ldots,b}_{m},c\lambda,\underbrace{2b,\ldots,2b}_{m},c\lambda,\underbrace{3b,\ldots,3b}_{m},\ldots).$
\item [\rm (viii)]
The row sequence $(T_{n,k})_{k=0}^n$ is a PF sequence for
$\{b,c,d,\lambda\}\subseteq \mathbb{R}^{\geq0}$.
 \end{itemize}
\end{prop}

\begin{proof}
By taking $a$ with $-a$ and $b$ with $-b$ in Proposition
\ref{prop+two+term+triangle-}, we have $[T_{n,k}]_{n,k}$ is the
exponential Riordan array $(g(t)e^{\lambda f(t)},f(t))$, where
$$g(t)=(1-abt)^{-d},\quad
f(t)=c\left[\frac{(1-abt)^{-\frac{1}{a}}-1}{b}\right].$$  So we
immediately get (i).

 It is easy to get that the compositional
inverse of $f(t)$ is
$$\bar{f}(t)=\frac{1-\left(1+\frac{bt}{c}\right)^{-a}}{ab}.$$ Clearly, both
$g(t)$ for $d\in \mathbb{N}$ and
$$1/\bar{f}'(t)=c\left(1+\frac{b}{c}t\right)^{a+1}$$ are P\'olya frequency
functions. Thus, for $d\in \mathbb{N}$, we can immediately get
$(b,\lambda)$-total positivity in (ii) by Theorem \ref{thm+ERA+TP}.
For $c=b$, applying Theorem \ref{thm+Sheffer+convolution+2} to the
exponential Riordan array $(g(t)e^{\lambda f(t)},f(t))$, we get the
results in (iii)-(vi).

For (vii), by taking $q\rightarrow c(q+\lambda)/b$ and $t\rightarrow
bt$, it suffices to prove for the reduced case $b=c=1$ and
$\lambda=0$ that
\begin{eqnarray*}
   \sum_{n\geq0}T_{n}(q)t^n
   & = &
   \cfrac{1}
         {1 \,-\, \alpha_{m} t
            \prod\limits_{i_1=1}^{m}
                 \cfrac{1}
            {1 \,-\, \alpha_{m+i_1} t
               \prod\limits_{i_2=1}^{m}
               \cfrac{1}
            {1 \,-\, \alpha_{m+i_1+i_2} t
               \prod\limits_{i_3=1}^{m}
               \cfrac{1}{1 - \cdots}
            }
           }
         }
%
%   \cfrac{1}{1 - \cfrac{\alpha_{k+m} t}
%            {\prod\limits_{i_1=1}^{m}
%               \Biggl( \vspace*{-1cm}\displaystyle
%                      1 - \cfrac{{\alpha_{k+m+i_1} t}}
%                               {\displaystyle \prod\limits_{i_2=1}^{m}
%               \biggl( 1 - \frac{\alpha_{k+m+i_1+i_2} t}
%                               {1 - \cdots}
%               \biggr)
%                               }
%               \Biggr)
%            }}
\end{eqnarray*}
with coefficients
$$(\alpha_{i})_{i\geq{m}}=(q,\underbrace{1,\ldots,1}_{m},q,\underbrace{2,\ldots,2}_{m},q,\underbrace{3,\ldots,3}_{m},\ldots).$$
For the reduced case, we have
$$\frac{1}{\bar{f}'(t)}=(1+t)^{m}.$$ Hence we immediately get the continued fraction expansion for the reduced case
by Theorem \ref{thm+Sheffer+convolution} (vi).

(viii) In order to prove that $(T_{n,k})_{k=0}^n$ is a PF sequence,
it suffices to prove that $T_n(q)$ has only real zeros. By the
recurrence relation (\ref{rec+four+}), we have
\begin{eqnarray}\label{rec+T+Der}
T_n(q)&=&\left[ab(n-1)+abd+c\lambda+cq\right]T_{n-1}(q)+b(q+\lambda)T'_{n-1}(q)
\end{eqnarray}
for $n\geq1$ and $T_{0}(q)=1$. In the following, we will prove a
stronger result that all zeros of $T_n(q)$ are real numbers and not
more than $-\lambda$. It is obvious for $n=0,1$. We assume that it
is true for $n-1$. Denote the zeros of $T_{n-1}(q)$ by $r_i$ for
$1\leq i\leq n-1$. It follows from the recurrence relation
(\ref{rec+T+Der}) that
$$sign[T_n(r_i)]=(-1)^i$$
for $1\leq i\leq n-1$. Note that coefficients of $T_n(q)$ are
nonnegative and $T_n(-\lambda)\geq0$. Then we get that $T_n(q)$ has
$n$ real zeros denoted by $z_1\geq z_2 \geq\ldots \geq z_n$ such
that
\begin{eqnarray*}
-\lambda\geq z_1\geq r_1\geq z_2 \geq r_2\geq\ldots \geq r_{n-1}\geq
z_n.
\end{eqnarray*}
 This completes the proof.
\end{proof}

\begin{rem}
For the array in (\ref{rec+four+}), we can give other sufficient
conditions for total positivity. For example, for the exponential
Riordan array $(g(t)e^{\lambda f(t)},f(t))$ in Proposition
\ref{prop+two+term+triangle}, it is easy to get
$$Z(t)=(abd+\lambda c+\lambda bt) \left(1+\frac{b}{c}t\right)^{a},
A(t)=(c+bt)\left(1+\frac{b}{c}t\right)^{a}.$$ By taking
$\varphi(t)=\left(1+\frac{b}{c}t\right)^{a}$ in Theorem
\ref{thm+Hakel}, we get its production matrix
\begin{eqnarray*}
P &=& \left(\prod_{i=1}^a\left[
  \begin{array}{ccccccc}
    1&& \\
    \frac{b}{c}&1&\\
    &\frac{2b}{c}&1&&\\
    &&\frac{3b}{c}&1&\\
 &&&\ddots&\ddots \\
  \end{array}
\right]\right)\Lambda\left[
\begin{array}{cccccc}
abd+c\lambda & c &  &  &\\
b\lambda & abd+c\lambda+b & 2c &\\
 & b\lambda & abd+c\lambda+2b & \ddots&\\
& & \ddots&\ddots  \\
\end{array}\right]\Lambda^{-1}.
\end{eqnarray*}
Note that the tridiagonal matrix $$\left[
\begin{array}{cccccc}
abd+c\lambda & c &  &  &\\
b\lambda & abd+c\lambda+b & 2c &\\
 & b\lambda & abd+c\lambda+2b & \ddots&\\
& & \ddots&\ddots  \\
\end{array}\right]$$ is TP for $b>c$ and $\frac{abd}{b-c}\geq\lambda>0$ (see \cite[Proposition 3.1]{Zhu202} for instance)
and is $(b,d)$-TP for $\lambda=0$ and $c>0$. Hence we get the
corresponding total positivity of $[T_{n,k}]_{n\geq0}$ and
$[T_{i+j}(q)]_{i,j}$ by Theorem \ref{thm+Hakel} (i).
\end{rem}

\begin{rem}
For $\lambda=0$ and $c=1$, we have many classical combinatorial
numbers as the special cases of $T_{n,k}$, {\it e.g.,} the Stirling
number of the second kind, the Carlitz¡¯s degenerate Stirling number
\cite{Car79}, the Howard¡¯s weighted degenerate Stirling number
\cite{How85},  the Todorov¡¯s number \cite{Tod88}, and the
Ahuja-Enneking¡¯s associated Lah number \cite{ND87}.
\end{rem}

Let $[\mathcal {T}_{n,k}]_{n,k}$ be defined by (\ref{rec+four+}).
For the reciprocal generalized Lah polynomial
$T_n^{*}(q)=q^nT_n(1/q)$ and the reciprocal generalized Lah numbers
$T^*_{n,k}=T_{n,n-k}$, by Proposition \ref{prop+two+term+triangle},
we immediately get the following reciprocal result.

\begin{prop}
Let $a\in \mathbb{N}$.  Then we have the following results.
\begin{itemize}
\item [\rm (i)]
The reciprocal generalized Lah triangle $[T^*_{n,k}]_{n,k\geq0}$
satisfies the next recurrence relation
\begin{eqnarray*}
T^*_{n,k}&=&c\,
T^*_{n-1,k}+\left[ab(n-1)+b(n-k)+abd+c\lambda\right]T^*_{n-1,k-1}+b\lambda(n-k+1)
T^*_{n-1,k-2}
\end{eqnarray*}
for $n,k\geq1$, where $T^*_{0,0}=1$.
\item [\rm (ii)]
The exponential generating function of $T^*_{n}(q)$ can be written
as
\begin{eqnarray*}
\sum_{n\geq0}T^*_{n}(q)\frac{t^n}{n!}=g(qt)\exp\left((q\lambda+1)f(qt)/q\right).
\end{eqnarray*}
\item [\rm (iii)]
For $c=b$, the sequence $(T^*_{n}(q))_{n\geq0}$ is
$(b,d,\lambda,q)$-SM and $3$-$(b,d,\lambda,q)$-log-convex. In
particular, $(T^*_{n,n})_{n\geq0}$ is $(b,d,\lambda)$-SM and
$3$-$(b,d,\lambda)$-log-convex.
\item [\rm (iv)]
Let $c=b$. The convolution $z_n=\sum_{k\geq0}T^*_{n,k}x_ky_{n-k}$
preserves Stieltjes moment property of sequences for
$\{b,d,\lambda\}\subseteq \mathbb{R}^{\geq0}$;
\item [\rm (v)]
Let $a+1=m$ and $d=0$. We have the $m$-branched Stieltjes-type
continued fraction expansion
\begin{eqnarray*}
   \sum_{n\geq0}T^*_{n}(q)t^n
   & = &
   \cfrac{1}
         {1 \,-\, \alpha_{m} t
            \prod\limits_{i_1=1}^{m}
                 \cfrac{1}
            {1 \,-\, \alpha_{m+i_1} t
               \prod\limits_{i_2=1}^{m}
               \cfrac{1}
            {1 \,-\, \alpha_{m+i_1+i_2} t
               \prod\limits_{i_3=1}^{m}
               \cfrac{1}{1 - \cdots}
            }
           }
         }
%
%   \cfrac{1}{1 - \cfrac{\alpha_{k+m} t}
%            {\prod\limits_{i_1=1}^{m}
%               \Biggl( \vspace*{-1cm}\displaystyle
%                      1 - \cfrac{{\alpha_{k+m+i_1} t}}
%                               {\displaystyle \prod\limits_{i_2=1}^{m}
%               \biggl( 1 - \frac{\alpha_{k+m+i_1+i_2} t}
%                               {1 - \cdots}
%               \biggr)
%                               }
%               \Biggr)
%            }}
\end{eqnarray*}
with coefficients
$$(\alpha_{i})_{i\geq{m}}=(c(1+q\lambda),\underbrace{bq,\ldots,bq}_{m},c(1+q\lambda),\underbrace{2bq,\ldots,2bq}_{m},c(1+q\lambda),\underbrace{3bq,\ldots,3bq}_{m},\ldots).$$
\item [\rm (iv)] Let $a+1=m$ and $d=0$. We have
\begin{eqnarray*}
   \sum_{n\geq0}T^*_{n,n}t^n
   & = &
   \cfrac{1}
         {1 \,-\, \alpha_{m} t
            \prod\limits_{i_1=1}^{m}
                 \cfrac{1}
            {1 \,-\, \alpha_{m+i_1} t
               \prod\limits_{i_2=1}^{m}
               \cfrac{1}
            {1 \,-\, \alpha_{m+i_1+i_2} t
               \prod\limits_{i_3=1}^{m}
               \cfrac{1}{1 - \cdots}
            }
           }
         }
%
%   \cfrac{1}{1 - \cfrac{\alpha_{k+m} t}
%            {\prod\limits_{i_1=1}^{m}
%               \Biggl( \vspace*{-1cm}\displaystyle
%                      1 - \cfrac{{\alpha_{k+m+i_1} t}}
%                               {\displaystyle \prod\limits_{i_2=1}^{m}
%               \biggl( 1 - \frac{\alpha_{k+m+i_1+i_2} t}
%                               {1 - \cdots}
%               \biggr)
%                               }
%               \Biggr)
%            }}
\end{eqnarray*}
with coefficients
$(\alpha_{i})_{i\geq{m}}=(c\lambda,\underbrace{b,\ldots,b}_{m},c\lambda,\underbrace{2b,\ldots,2b}_{m},c\lambda,\underbrace{3b,\ldots,3b}_{m},\ldots).$
\item [\rm (v)]
The row sequence $(T^*_{n,k})_{k=0}^n$ is a PF sequence for
$\{b,c,d,\lambda\}\subseteq \mathbb{R}^{\geq0}$.\end{itemize}
\end{prop}

 Similar to Proposition \ref{prop+two+term+triangle-k}, we also
have the next result.
\begin{prop}
Let $a\in \mathbb{N}$ and $c\in \mathbb{R}^{+}$. Assume that a
triangular array $[\widetilde{T}_{n,k}]_{n,k\geq0}$ satisfies the
following recurrence relation
\begin{eqnarray*}
\widetilde{T}_{n,k}&=&c
k\widetilde{T}_{n-1,k-1}+\left[ab(n-1)+bk+abd+c\lambda\right]\widetilde{T}_{n-1,k}+b\lambda
\widetilde{T}_{n-1,k+1}
\end{eqnarray*}
for $n,k\geq1$, where $\widetilde{T}_{0,0}=1$. Let
$\widetilde{T}_n(q)=\sum_{k\geq0}\widetilde{T}_{n,k}q^k$. Then we
have the following results.
\begin{itemize}
\item [\rm (i)]
 An explicit formula of $\widetilde{T}_{n,k}$ can be written as
\begin{eqnarray*}
\widetilde{T}_{n,k}&=&a^n\sum_{j\geq0}\sum_{i=0}^{k+j}\frac{(-1)^{n+k+j-i}b^{n-k-j}c^{k+j}\lambda^j}{j!}\binom{k+j}{i}\left(-\frac{i+ad}{a}\right)_n
\end{eqnarray*}
for $n,k\geq1$.
 \item [\rm (ii)]
 For $c=b$, the lower-triangular matrix $[\widetilde{T}_{n,k}]_{n,k\geq0}$ is $(b,d,\lambda)$-TP.
 \item [\rm (iii)]
For $d=0$, the exponential generating function is
\begin{eqnarray*}
1+\sum_{n\geq1}\widetilde{T}_n(q)\frac{t^n}{n!}=\frac{1}{1+\frac{c}{b}\left[1-(1-abt)^{-\frac{1}{a}}\right](q+\lambda)}.
\end{eqnarray*}
 \item [\rm (iv)]
Let $c=b$. The convolution
$z_n=\sum_{k\geq0}\widetilde{T}_{n,k}x_ky_{n-k}$ preserves Stieltjes
moment property of sequences for $\{b,d,\lambda\}\subseteq
\mathbb{R}^{\geq0}$.
 \end{itemize}
\end{prop}

\subsection{Rook polynomials and signless Laguerre polynomials}
Let $\mathfrak{S}_n(q)$ denote the rook polynomial of a square of
side $n$. It is well-known that it is given by
$$\mathfrak{S}_n(q)=\sum_{k=0}^n\binom{n}{k}^2k!q^k,$$
see \cite[Chapter 3. Problems 18]{Ri} for instance, and it can also
be viewed as the matching polynomial of the complete bipartite graph
$K_{n,n}$. It has only real zeros in terms of the rook theory or
matching polynomials. The rook polynomials $\mathfrak{S}_n(q)$ form
a strongly $q$-log-convex sequence \cite{ZS15}. Recently, Wang and
Zhu \cite{WZ16} proved that the transformation
$$z_n=\sum_{k=0}^{n}\binom{n}{k}^2x_k$$
preserves the Stieltjes moment property for real numbers. Note for
fixed $q>0$ that $(k!q^k)_{k\geq0}$ is a Stieltjes moment sequence.
Thus we get for $n\geq0$ that
$$\mathfrak{S}_n(q)=\sum_{k=0}^n\binom{n}{k}^2k!q^k$$
form a Stieltjes moment sequence for fixed $q>0$. These support the
following stronger conjecture of Sokal.

\begin{conj}\emph{\cite{Sok19}}\label{conj+sokal}
The sequence $(\mathfrak{S}_n(q))_{n\geq0}$ is $q$-Stieltjes moment,
that is to say that the Hankel matrix
$[\mathfrak{S}_{i+j}(q)]_{i,j}$ is $q$-TP.
\end{conj}

As an application of our results, we will demonstrate Conjecture
\ref{conj+sokal}.
\begin{prop}\label{prop+rook}
The sequence $(\mathfrak{S}_n(q))_{n\geq0}$ is $q$-Stieltjes moment.
\end{prop}

The rook polynomials $\mathfrak{S}_n(q)$ are closely related to the
famous Laguerre polynomials. For the Laguerre polynomial
$L^{(\alpha)}_n(q)$ with $\alpha\geq-1$ (see \cite{AAR99} for
instance), its exponential generating function is
\begin{equation}\label{Exp+GF+Lag}
\sum_{n\geq0}L^{(\alpha)}_n(q)\frac{t^n}{n!}=\frac{1}{(1-t)^{\alpha+1}}\exp\left(\frac{qt}{t-1}\right).
\end{equation}
In addition, it can be given by
$$L^{(\alpha)}_n(q)=\sum_{k=0}^n\binom{n+\alpha}{n-k}\frac{n!}{k!}(-q)^k$$
for $n\geq0$. Let $\widetilde{L}^{(\alpha)}_n(q)=L^{(\alpha)}_n(-q)$
for $n\geq0$. It is obvious for $\alpha=0$ that
$$\mathfrak{S}_n(q)=q^n\widetilde{L}^{(0)}_n(1/q).$$
In addition, for $\alpha=-1$, the Lah polynomial $\mathcal
{L}_n(q)=\widetilde{L}^{(-1)}_n(q)$. For signless Laguerre
polynomials $\widetilde{L}^{(\alpha)}_n(q)$, we have the following
result, which also implies Proposition \ref{prop+rook}.

\begin{prop}\label{prop+Lag}
Let $\alpha\geq-1$ and
$\widetilde{L}^{(\alpha)^{*}}_n(q)=q^n\widetilde{L}^{(\alpha)}_n(1/q)$.
Then we have
 \begin{itemize}
\item [\rm (i)]
both $(\widetilde{L}^{(\alpha)}_n(q))_{n\geq0}$ and
$(\widetilde{L}^{(\alpha)^{*}}_n(q))_{n\geq0}$ are $q$-Stieltjes
moment and $3$-$q$-log-convex;
\item [\rm (ii)]
the triangular matrix $[\binom{n+\alpha}{n-k}\frac{n!}{k!}]_{n,k}$
is TP;
\item [\rm (iii)]
the convolution
$$z_n=\sum_{k\geq0}\binom{n+\alpha}{n-k}\frac{n!}{k!}x_ky_{n-k}$$
preserves the SM property.
 \end{itemize}
\end{prop}

\begin{proof}
By (\ref{Exp+GF+Lag}), we have
\begin{equation*}
\sum_{n\geq0}\widetilde{L}^{(\alpha)}_n(q)\frac{t^n}{n!}=\frac{1}{(1-t)^{\alpha+1}}\exp\left(\frac{qt}{1-t}\right).
\end{equation*}
Let us consider the exponential Riordan array $(g(t),f(t))$, where
$$g(t)=\frac{1}{(1-t)^{\alpha+1}}\exp\left(\frac{qt}{1-t}\right)=(1+f(t))^{\alpha+1}\exp(qf(t)),\quad f(t)=\frac{t}{1-t}.$$
Hence we get that $$\frac{1}{(1+t)\bar{f}'(t)}=1+t$$ is a PF
function. So by Proposition \ref{prop+two+term+triangle} or Theorem
\ref{thm+Sheffer+convolution+2}, the desired results are immediate.
\end{proof}

\subsection{Idempotent numbers}

It is well-known that idempotent numbers
$$I_{n,k}=\binom{n}{k}k^{n-k}$$ satisfy
$$\exp(qt\exp(t))=\sum_{n\geq0}
\sum_{k=0}^nI_{n,k}q^k \frac{t^n}{n!}.$$ The inverse of the
idempotent triangle $[\binom{n}{k}k^{n-k}]_{n,k\geq0}$ is
$\left[(-1)^{n-k}\binom{n-1}{k-1}n^{n-k}\right]_{n,k\geq1}$, where
$\binom{n-1}{k-1}n^{n-k}$ counts the number of rooted labeled trees
on $n+1$ vertices with a root degree $k$, see \cite[A137452]{Slo}.
It is known that the compositional inverse of $t\exp(t)$ is the
Lambert function $W(t)=\sum_{n\geq1}\frac{(-n)^{n-1}t^n}{n!}$ (This
can be obtained in terms of the Lagrange inversion formula, see
\cite[p. 152]{Com74} for instance). Thus
$$\exp(-qW(-t))=1+\sum_{n\geq1}
\sum_{k=1}^n\binom{n-1}{k-1}n^{n-k}q^k \frac{t^n}{n!}$$ and the
compositional inverse of $-W(-t)$ is $t\exp(-t)$. Obviously,
$\frac{1}{(t\exp(-t))'}=\frac{\exp(t)}{1-t}$ is a PF function. Then
the following is immediate from Theorem
\ref{thm+Sheffer+convolution}.
\begin{prop}
\begin{itemize}
 \item [\rm (i)]
 Matrices $\left[\binom{n}{k}k^{n-k}\right]_{n,k\geq0}$ and
 $\left[\binom{n-1}{k-1}n^{n-k}\right]_{n,k\geq1}$ are TP.
 \item [\rm (ii)]
 Matrices $\left[\binom{n}{k}k^{n-k}k!\right]_{n,k\geq0}$ and
 $\left[\binom{n-1}{k-1}n^{n-k}k!\right]_{n,k\geq1}$ are TP.
  \item [\rm (iii)]
The sequence $(\mathcal {I}_n(q))_{n\geq0}$ is $q$-SM and
$3$-$q$-log-convex, where $ \mathcal
{I}_n(q)=\sum_{k=1}^n\binom{n-1}{k-1}n^{n-k}q^k$.
\item [\rm (iv)]
The sequence $((1+qn)^{n-1})_{n\geq1}$ are $q$-SM and
$3$-$q$-log-convex.
\item [\rm (v)]
The sequence $(n^{n-1})_{n\geq1}$ is SM.
 \item [\rm (vi)]
Both convolutions
$z_n=\sum_{k\geq1}\binom{n-1}{k-1}n^{n-k}x_ky_{n-k}$ and
$z_n=\sum_{k\geq0}\binom{n}{k}(n+1)^{n-k}x_ky_{n-k}$ preserve
Stieltjes moment property of sequences.
 \end{itemize}
\end{prop}

\subsection{Some numbers related to binomial coefficients}
It is well-known that the Pascal triangle
$[\binom{n}{k}]_{n,k\geq0}$ is TP, see \cite{Kar68} for instance.
Let $C_k = \binom{n+ck}{m+dk}$, where $n \geq m$. Then for integers
$d>c> 0$, the Toeplitz matrix $[C_{i-j}]_{i,j\geq0}$ of the finite
sequence $(C_k)_k$ is TP, which was conjectured in \cite{SW08} and
proved in \cite{Y09}. Recently, we also proved for real numbers $c
\geq d> 0$ that the Hankel matrix $[C_{i+j}]_{i,j\geq0}$ of the
infinite sequence $(C_k)_k$ is TP \cite{Zhu191}. In the following,
we will present more results for total positivity related to
binomial coefficients.
\begin{prop}
Let $c$ and $d$ be integers and $\{m,n\}\subseteq \mathbb{N}$. Then
\begin{itemize}
 \item [\rm (i)]
for $d>0$ and $d\geq c$, the low-triangular matrices
$\left[\binom{n}{k}\binom{n+ck}{m+dk}(n-k)!\right]_{n,k\geq0}$ and
$\left[\binom{n+ck}{m+dk}\right]_{n,k\geq0}$ are TP;
 \item [\rm (ii)]
for $d\leq-1$ and $c\in \mathbb{N}$, the low-triangular matrices
$\left[\binom{m}{k}\binom{n+ck}{m+dk}(m-k)!\right]_{m,k\geq0}$ and
$\left[\binom{n+ck}{m+dk}\right]_{m,k\geq0}$ are TP.
 \end{itemize}
\end{prop}
\begin{proof}
By Lemma \ref{lem+factor+TP}, it suffices to prove that matrices
$\left[\binom{n}{k}\binom{n+ck}{m+dk}(n-k)!\right]_{n,k\geq0}$ in
(i) and
$\left[\binom{m}{k}\binom{n+ck}{m+dk}(m-k)!\right]_{m,k\geq0}$ in
(ii) are TP. Note that
\begin{eqnarray*}
\binom{n}{k}\binom{n+ck}{m+dk}(n-k)!&=&\frac{n!}{k!}\binom{n+ck}{(n-m)+(c-d)k}\\
&=&\frac{n!}{k!}\binom{-m-dk-1}{(n-m)+(c-d)k}(-1)^{(n-m)+(c-d)k}\\
&=&\frac{n!}{k!}[t^{(n-m)+(c-d)k}]\frac{1}{(1-t)^{m+1+dk}}\\
&=&\frac{n!}{k!}[t^{n}]\frac{t^m}{(1-t)^{m+1}}\left(\frac{t^{d-c}}{(1-t)^{d}}\right)^k\\
&=&\left(\frac{t^m}{(1-t)^{m+1}},\frac{t^{d-c}}{(1-t)^{d}}\right),\\
\end{eqnarray*}
\begin{eqnarray*}
\binom{m}{k}\binom{n+ck}{m+dk}(m-k)!&=&\frac{m!}{k!}\binom{n+ck}{m+dk}\\
&=&\frac{m!}{k!}[t^{m+dk}](1+t)^{n+ck}\\
&=&\frac{m!}{k!}[t^m](1+t)^n\left(t^{-d}(1+t)^{c}\right)^k\\
&=&\left((1+t)^n,t^{-d}(1+t)^{c}\right),
\end{eqnarray*}
where $\frac{t^m}{(1-t)^{m+1}}$, $\frac{t^{d-c}}{(1-t)^{d}}$,
$(1+t)^n$ and $t^{-d}(1+t)^{c}$ are P\'olya frequency functions
under our conditions. Thus, it follows from Theorem \ref{thm+ERA+TP}
and Remark \ref{rem+TP} that
$\left[\binom{n}{k}\binom{n+ck}{m+dk}(n-k)!\right]_{n,k\geq0}$ in
(i) and
$\left[\binom{m}{k}\binom{n+ck}{m+dk}(m-k)!\right]_{m,k\geq0}$ in
(ii) are TP. The proof is complete.
\end{proof}

%%%%%%%%%%%%%%% References
\section{Acknowledgements}
The author would like to thank the anonymous reviewer for many
valuable remarks and suggestions to improve the original manuscript.

\end{document}